\documentclass[11pt]{article}

\AtEndDocument{\bigskip{\footnotesize%
\textsc{Inria Saclay-Ile-de-France Alan Turing Bldg,1 rue Honoré d'Estienne d'Orves, 91120 Palaiseau, France.}
  \par \textit{E-mail address:} \texttt{nicolas.berkouk@inria.fr} \par   
  \noindent \textsc{Universit\'e de Paris, CRESS, INSERM, INRA, F-75004 Paris France.} \par
  \textit{E-mail address:} \texttt{francois.petit@u-paris.fr} 
  }}

\usepackage{amsmath,amssymb,amsthm,mathrsfs}
\usepackage[all,2cell]{xy}
\usepackage{graphicx}
\usepackage{enumerate} 
\usepackage{fixme}
\usepackage{lmodern}
\usepackage{dsfont}
\usepackage{color}
\usepackage{hyperref}
\usepackage[margin=3cm]{geometry}

\DeclareMathOperator{\Fun}{Fun}
\DeclareMathOperator{\colim}{colim}

\DeclareMathOperator{\Int}{Int}

\DeclareMathOperator{\RHom}{RHom}

\DeclareMathOperator{\Hom}{Hom}

\DeclareMathOperator{\id}{id}
\DeclareMathOperator{\opp}{op}
\DeclareMathOperator{\Mod}{Mod}
\DeclareMathOperator{\Ch}{Ch}

\DeclareMathOperator{\Hn}{H}
\DeclareMathOperator{\Mor}{Mor}

\DeclareMathOperator{\Supp}{Supp}

\DeclareMathOperator{\Ker}{Ker}

\DeclareMathOperator{\Op}{Op}
\DeclareMathOperator{\Inter}{Inter}

\begin{document}

\theoremstyle{plain} 
\newtheorem{theorem}{Theorem}[section]
\newtheorem{corollary}[theorem]{Corollary}
\newtheorem{proposition}[theorem]{Proposition}
\newtheorem{lemma}[theorem]{Lemma}
\theoremstyle{definition} 
\newtheorem{definition}[theorem]{Definition}
\newtheorem{example}[theorem]{Example}
\newtheorem{examples}[theorem]{Examples}
\newtheorem{remark}[theorem]{Remark}
\newtheorem{notation}[theorem]{Notations}
\newtheorem{fact}[theorem]{Fact}

\numberwithin{equation}{section}

\newcommand{\pinv}{p_\ast^{\C^\times}\hspace{-0.2em}}
\newcommand{\On}[1]{\mathcal{O}_{#1}}
\newcommand{\En}[1]{\mathcal{E}_{#1}}
\newcommand{\Fn}[1]{\mathcal{F}_{#1}} 
\newcommand{\tFn}[1]{\mathcal{\tilde{F}}_{#1}}
\newcommand{\hum}[1]{hom_{\mathcal{A}}({#1})}
\newcommand{\hcl}[2]{#1_0 \lbrack #1_1|#1_2|\ldots|#1_{#2} \rbrack}
\newcommand{\hclp}[3]{#1_0 \lbrack #1_1|#1_2|\ldots|#3|\ldots|#1_{#2} \rbrack}
\newcommand{\partiel}[2]{\dfrac{\partial #1}{\partial #2}}
\newcommand{\catMod}{\mathsf{Mod}}
\newcommand{\Der}{\mathrm{D}}
\newcommand{\Ds}{D_{\mathbb{C}}}
\newcommand{\DG}{\mathsf{D}^{b}_{dg,\mathbb{R}-\mathsf{C}}(\mathbb{C}_X)}
\newcommand{\lI}{[\mspace{-1.5 mu} [}
\newcommand{\rI}{] \mspace{-1.5 mu} ]}
\newcommand{\Ku}[2]{\mathfrak{K}_{#1,#2}}
\newcommand{\iKu}[2]{\mathfrak{K^{-1}}_{#1,#2}}
\newcommand{\Be}{B^{e}}
\newcommand{\op}[1]{#1^{\opp}}
\newcommand{\Tt}{\mathcal{T}}
\newcommand{\N}{\mathbb{N}}
\newcommand{\Ab}[1]{#1/\lbrack #1 , #1 \rbrack}
\newcommand{\Du}{\mathbb{D}}
\newcommand{\C}{\mathbb{C}}
\newcommand{\R}{\mathbb{R}}
\newcommand{\Z}{\mathbb{Z}}
\newcommand{\V}{\mathbb{V}}
\newcommand{\W}{\mathbb{W}}
\newcommand{\al}{\mathfrak{a}}
\newcommand{\w}{\omega}
\newcommand{\K}{\mathcal{K}}
\newcommand{\Hoc}{\mathcal{H}\mathcal{H}}
\newcommand{\env}[1]{{\vphantom{#1}}^{e}{#1}}
\newcommand{\eA}{{}^eA}
\newcommand{\eB}{{}^eB}
\newcommand{\eC}{{}^eC}
\newcommand{\cA}{\mathcal{A}} 
\newcommand{\cB}{\mathcal{B}}
\newcommand{\cD}{\mathcal{D}}
\newcommand{\cR}{\mathcal{R}}
\newcommand{\cI}{\mathcal{I}}
\newcommand{\cL}{\mathcal{L}}
\newcommand{\cO}{\mathcal{O}}
\newcommand{\cM}{\mathcal{M}}
\newcommand{\cN}{\mathcal{N}}
\newcommand{\cK}{\mathcal{K}}
\newcommand{\cC}{\mathcal{C}}
\newcommand{\cF}{\mathcal{F}}
\newcommand{\cG}{\mathcal{G}}
\newcommand{\cP}{\mathcal{P}}
\newcommand{\cQ}{\mathcal{Q}}
\newcommand{\cU}{\mathcal{U}}
\newcommand{\cE}{\mathcal{E}}
\newcommand{\cS}{\mathcal{S}}
\newcommand{\cT}{\mathcal{T}}
\newcommand{\D}{\text{D}}
\newcommand{\Obj}{\text{Obj}}
\newcommand{\chE}{\widehat{\mathcal{E}}}
\newcommand{\cW}{\mathcal{W}}
\newcommand{\chW}{\widehat{\mathcal{W}}}
\newcommand{\Hper}{\Hn^0_{\textrm{per}}}
\newcommand{\Dper}{\Der_{\mathrm{perf}}}
\newcommand{\Yo}{\textrm{Y}}
\newcommand{\gqcoh}{\mathrm{gqcoh}}
\newcommand{\coh}{\mathrm{coh}}
\newcommand{\cc}{\mathrm{cc}}
\newcommand{\qcc}{\mathrm{qcc}}
\newcommand{\gd}{\mathrm{gd}}
\newcommand{\qcoh}{\mathrm{qcoh}}
\newcommand{\lcl}{\mathrm{lcl}}
\newcommand{\fin}{\mathrm{fin}}
\newcommand{\obplus}[1][i \in I]{\underset{#1}{\overline{\bigoplus}}}
\newcommand{\Lte}{\mathop{\otimes}\limits^{\rm L}}
\newcommand{\te}{\mathop{\otimes}\limits^{}}
\newcommand{\btimes}{\mathop{\boxtimes}\limits^{}}
\newcommand{\pt}{\textnormal{pt}}
\newcommand{\A}[1][X]{\cA_{{#1}}}
\newcommand{\dA}[1][X]{\cC_{X_{#1}}}
\newcommand{\conv}[1][]{\mathop{\circ}\limits_{#1}}
\newcommand{\sconv}[1][]{\mathop{\ast}\limits_{#1}}
\newcommand{\ldetens}{\overset{\mathnormal{L}}{\underline{\boxtimes}}}
\newcommand{\br}{\bigr)}
\newcommand{\bl}{\bigl(}
\newcommand{\sC}{\mathscr{C}}
\newcommand{\ucat}{\mathbf{1}}
\newcommand{\ubtimes}{\underline{\boxtimes}}
\newcommand{\uLte}{\mathop{\underline{\otimes}}\limits^{\rm L}} 
\newcommand{\Lp}{\mathrm{L}p}
\newcommand{\pder}[3][]{\frac{\partial^{#1}#2}{\partial{#3}}}
\newcommand{\reg}{\mathrm{reg}}
\newcommand{\sing}{\mathrm{sing}}
\newcommand{\fExt}{\mathcal{E}xt}
\newcommand{\fTor}{\mathcal{T}or}
\newcommand{\fEnd}{\mathcal{E}nd}
\newcommand{\dL}{\mathrm{L}}
\newcommand{\fgd}{\mathrm{fgd}}
\newcommand{\Zl}{Z}
\newcommand{\subtageq}[1]{\stepcounter{equation} \tag*{$(\arabic{section}.\arabic{equation})_#1$}}
\newcommand{\wtmu}{\widetilde{\mu}}
\newcommand{\lexp}{\,^{l}\!}


\newcommand{\vvert}{\Vert}
\newcommand{\tw}[1]{\widetilde{#1}}
\newcommand{\Derb}{\mathrm{D}^{\mathrm{b}}}
\newcommand{\cor}{{\bf k}}


\newcommand{\oim}[1]{{#1}_*}
\newcommand{\eim}[1]{{#1}_!}
\newcommand{\roim}[1]{\mathrm{R}{#1}_*}
\newcommand{\reim}[1]{\mathrm{R}{#1}_!}
\newcommand{\opb}[1]{#1^{-1}}
\newcommand{\epb}[1]{#1^{!}}
\newcommand{\spb}[1]{#1^{*}}
\newcommand{\lspb}[1]{\LL{#1}^{*}}
\newcommand{\deim}[1]{{#1}_{D}}
\newcommand{\dopb}[1]{#1^{D}}
\newcommand{\popb}[1]{#1^{\dag}}

\newcommand{\tens}[1][]{\mathbin{\otimes_{\raise1.5ex\hbox to-.1em{}{#1}}}}
\newcommand{\etens}{\mathbin{\boxtimes}}
\newcommand{\cl}{\colon}
\newcommand{\Eph}{\textnormal{Eph}}

\title{Ephemeral persistence modules and distance comparison}
\author{Nicolas Berkouk, Fran\c{c}ois Petit\footnote{The author has been  supported in the frame of the OPEN scheme of the Fonds National de la Recherche (FNR) with the project QUANTMOD O13/570706 and aknowledged also the support of the Idex “Universit\'e de Paris 2019”}}
\date{}

\maketitle

\begin{abstract}
We provide a definition of ephemeral multi-persistent modules and prove that the quotient category of persistent modules by the ephemeral ones is equivalent to the category of $\gamma$-sheaves. In the case of one-dimensional persistence, our definition agrees with the usual one showing that the observable category and the category of $\gamma$-sheaves are equivalent. We also establish isometry theorems between the category of persistent modules and $\gamma$-sheaves both endowed with their interleaving distance. Finally, we compare the interleaving and convolution distances. 
\end{abstract}

\tableofcontents

\section{Introduction}

One of the initial motivation of persistent homology was to provide a mean to estimate the topology of space from a finite noisy sample of itself. Persistent homology and more generally the concept of persistence have since been developed and have spread among many areas of mathematics, such as representation theory, symplectic topology and applied topology \cite{Asano2017,Botnan2018, Polterovich2019}.

Though persistence theory is well understood in the one-parameter case (see for instance \cite{Oudot2015} for an extensive exposition of the theory and its applications), its generalization to the multi-parameter case remains less understood, yet is important for applications \cite{Lesnick}. The first approach to study the category of multi-parameter persistence modules was with an eye coming from algebraic geometry and representation theory \cite{Carlsson2009}. Roughly speaking, the idea was to consider persistence modules as graded-modules over a polynomial ring. This allowed to link the theory of persistence with more classical areas of mathematics and allowed to show that a complete classification of persistence modules with more than one parameter is impossible. Nevertheless, one thing not to be forgotten is that the category of persistence modules is naturally endowed with the interleaving distance. Having applications in mind, one is more interested in computing the distance between two persistence modules, than to explicit the structural difference between those.

Sheaf theoretic methods have been recently introduced to study persistent homology. They first appeared in the work of J. Curry \cite{Curry2014}. In recent times, M. Kashiwara and P. Schapira in \cite{Kashiwara2018a, Kashiwara2018} introduced derived sheaf-theoretic methods in persistent homology.  Persistence homology studies filtered or multi-filtered topological spaces. The filtrations are indexed by the elements of an ordered vector space $\V$. The choice of the order is equivalent to the choice of a closed convex proper cone $\gamma \subset \V$. Hence, the idea underlying both approaches is to endow $\V$ with a topology depending on this cone. Whereas J. Curry's approach relies on Alexandrov's topology, M. Kashiwara and P. Schapira's approach is based on the $\gamma$-topology which was introduced by the same authors in \cite{KS90}. The goal of this paper is to compare these two approaches. A key feature of persistence theory is that the various versions of the space of persistent modules can be endowed with pseudo-distances. We focus our attention on two main types of pseudo-distances: the interleaving distances studied by several authors among which \cite{Chazal2016a,deSilva2018,Lesnick2012,Lesnick2015} and the convolution distance introduced in \cite{Kashiwara2018a} and studied in detail in the one-dimensional case in \cite{Berkouk}. Besides comparing the various categories of sheaves used in persistence theory (and especially multiparameter persistent homology), we establish isometry theorems between these categories endowed with their respective distances.

To compare Alexandrov sheaves and $\gamma$-sheaves, we first study morphisms of sites between the Alexandrov and the $\gamma$-topology. We precise the results of \cite[Section 1.4]{Kashiwara2018a} by introducing two morphisms of sites $\alpha \colon \V_\gamma \to \V_\al$ and $\beta \colon \V_\al \to \V_\gamma$ where $\V_\al$ denotes the vector space $\V$ endowed with the Alexandrov topology while $\V_\gamma$ designates $\V$ endowed with the $\gamma$ topology. This provides us with three distinct functors $\oim{\alpha}, \opb{\beta} \colon \Mod(\cor_{\V_\gamma}) \to \Mod(\cor_{\V_\al})$ and $\oim{\beta}=\opb{\alpha} \colon \Mod(\cor_{\V_\al}) \to \Mod(\cor_{\V_\gamma})$ where $\Mod(\cor_{\V_\gamma})$ (resp.\ $\Mod(\cor_{\V_\al})$) is the category of sheaves of $\cor$-modules on $\V_\gamma$ (resp.\ $\V_\al$). The properties of these functors allow us to define a well-behaved notion of ephemeral modules in arbitrary dimensions (Definition \ref{def:eph}). They correspond to Alexandrov sheaves which vanish when evaluated on open subsets of the $\gamma$-topology. In dimension one, our notion of ephemeral module coincides with the one introduced in \cite{Chazal2016a} and further studied in \cite{Chazal2016b} and \cite{Berkouk}. Then, we show that the quotient of the category $\Mod(\cor_{\V_\al})$ by its subcategory of ephemeral modules is equivalent to the category $\Mod(\cor_{\V_\gamma})$ (Theorem \ref{thm:GammaSh_SerreQ}). Specializing again our results to the situation where $\dim \V=1$, we obtain a canonical equivalence of categories between the observable category $\mathbf{Ob}$ of \cite{Chazal2016b} and the category $\Mod(\cor_{\V_\gamma})$ (Corollary \ref{cor:obgam}). This provides a natural description of the category of observable modules and  highlights the significance of the theory of $\gamma$-sheaves for studying persistent homology. We extend all these results to the derived setting. 

We establish an isometry theorem between the category of Alexandrov sheaves and $\gamma$-sheaves on $\V$ endowed with their respective interleaving distance (Theorem \ref{thm:isom} and Corollary \ref{cor:isomab}). Note that our approach does not rely on a structure theorem for persistence modules (as such theorem is not available in arbitrary dimension) but on the properties of the morphisms of sites $\alpha$ and $\beta$. We also study the properties of ephemeral modules with respect to the notion of interleaving and show that they correspond to modules which are interleaved with zero in all the directions allowed by the Alexandrov topology. This shows that the notion of ephemeral module is more delicate in higher dimensions than in dimension one. This being essentially due to the fact that in dimension one the boundary of the cone associated with the usual order on $\R$ is of dimension zero. 

Finally, we study the relation between the interleaving and the convolution distances on the category of $\gamma$-sheaves. The convolution distance depends on the choice of a norm on $\V$. Given an interleaving distance with respect to a vector $v$ in the interior of the cone $\gamma$, we introduce a preferred norm (see formula \eqref{eq:normcool}) and show that, under a mild assumption on the persistence modules considered, the convolution distance associated with this norm and the interleaving distance associated with $v$ are equal (Corollary \ref{cor:gamconv}).\\

\noindent \textbf{Acknowledgment:} The authors are grateful to Pierre Schapira for his scientific advice. The second author would like to thank Yannick Voglaire for several useful conversations. Both authors would like to thank the IMA for its excellent working conditions as part of this work was done during the Special Workshop on Bridging Statistics and Sheaves.  

\section{Sheaves on $\gamma$ and Alexandrov topology}

\subsection{$\gamma$ and Alexandrov topology}

\subsubsection{$\gamma$-topology}
Following \cite{Kashiwara2018a}, we briefly review the notion of $\gamma$-topology. We refer the reader to \cite{KS90} for more details.

Let $\V$ be a finite dimensional real vector space. We write $s$ for the sum map $s \colon \V \times \V \to \V$, $(x,y) \mapsto x+y$ and $a \colon x \mapsto -x$ for the antipodal map. If $A$ is a subset of $\V$, we write $A^a$ for the antipodal of $A$, that is the subset $\lbrace x \in \V \vert -x \in A \rbrace $.

A subset $C$ of the vector space $\V$ is a cone if
\begin{enumerate}[(i)]
 \item $0 \in \V$,
 \item $\R_{>0} \, C \subset C$.
\end{enumerate}
We say that a convex cone $C$ is proper if $C^a \cap C =\lbrace 0 \rbrace$.

Given a cone $C \subset \V$, we define its polar cone $C^\circ$ as the cone of $\V^\ast$
\begin{equation*}
C^\circ=\lbrace \xi \in \V^\ast \, | \, \forall v \in C, \langle \xi, v \rangle \geq 0 \rbrace.
\end{equation*} 
From now on, $\gamma$ denotes a
\begin{equation}\label{hyp:cone}
\textnormal{ \textit {closed proper convex cone with non-empty interior.}}
\end{equation}
We still write $\V$ for the vector space $\V$ endowed with the usual topology.

We say that a subset $A$ of $\V$ is \textit{$\gamma$-invariant} if $A=A+ \gamma$. The set of $\gamma$-invariant open subsets of $\V$ is a topology on $\V$ called the $\gamma$-topology. We denote by $\V_\gamma$ the vector space $\V$ endowed with the $\gamma$-topology. We write $\phi_\gamma \colon \V \to \V_{\gamma}$ for the continuous map whose underlying function is the identity.

If $A$ is a subset of $\V$, we write $\Int(A)$ for the interior of $A$ in the usual topology of $\V$. 

\begin{lemma}
Let $U$ be a $\gamma$-open then $U= \bigcup_{x \in U}  x+ \Int(\gamma)$.
\end{lemma}

\begin{proof}
The proof is left to the reader.
\end{proof}

\subsubsection{$\gamma$-sheaves}

In this section, following \cite{KS90}, we recall the notion of $\gamma$-sheaves and results borrowed to \cite{Kashiwara2018a} and \cite{GS2014}.

\begin{notation}
Let $\cor$ be a field. For a topological space $X$, we denote by $\cor_{X}$ the constant sheaf on $X$ with coefficient in $\cor$ and write $\Mod(\cor_{X})$ for the abelian category of $\cor_{X}$-modules, $\Ch(\cor_X)$ for the abelian category of chain complexes of $\cor_X$-modules, $\Der(\cor_{X})$ for the unbounded derived category of $\Mod(\cor_{X})$ and $\Derb(\cor_{X})$ for its bounded derived category. That is the full subcategory of $\Der(\cor_{X})$ whose objects are the $F \in \Der(\cor_{X})$ such that there exists $n \in \N$ such that for every $k \in \Z$ with $|k| \geq n$, $\Hn^k(F) \simeq 0$.
\end{notation}

We now state a result due to M. Kashiwara and P. Schapira that says that the bounded derived category of $\gamma$-sheaves is equivalent to a subcategory of  the usual bounded derived category of sheaves $\Derb(\cor_\V)$. This subcategory can be characterized by a microsupport condition. We refer the reader to \cite[Chapter V]{KS90} for the definition and properties of the microsupport.

Following \cite{Kashiwara2018a}, we set
\begin{align*}
\Derb_{\gamma^{\circ,a}}(\cor_\V)=\lbrace F \in \Derb(\cor_\V) \, \vert \, \SS(F) \subset \gamma^{\circ,a} \rbrace,\\
\Mod_{\gamma^{\circ,a}}(\cor_\V)=\Mod(\cor_\V)\cap\Derb_{\gamma^{\circ,a}}(\cor_\V).
\end{align*}

\begin{theorem}[{\cite[Theorem 1.5]{Kashiwara2018a}}]
Let $\gamma$ be a proper closed convex cone in $\V$. The functor $\roim{\phi_\gamma} \colon \Derb_{\gamma^{\circ,a}}(\cor_\V) \to \Derb(\cor_{\V_\gamma})$ is an equivalence of triangulated categories with quasi-inverse $\opb{\phi_\gamma}$.
\end{theorem}

\begin{corollary}
The functor $\oim{\phi_\gamma} \colon \Mod_{\gamma^{\circ,a}}(\cor_\V) \to \Mod(\cor_{\V_\gamma})$ is an equivalence of categories with quasi-inverse $\opb{\phi_\gamma}$.
\end{corollary}

Consider the following maps: 
\begin{equation*}
s : \V \times \V \to \V, \quad s(x,y) = x + y\\  
\end{equation*}
\begin{equation*}
q_i : \V \times \V \to \V \; (i=1,2) ~~~q_1(x,y) = x,~q_2(x,y) = y
\end{equation*}

Let $F$ and $G$ in $\Derb(\cor_V)$, we set

\begin{equation*}
F\sconv[np] G = \roim{s}(\opb{q_1} F \otimes \opb{q_2} G)=\roim{s}( F \boxtimes  G).
\end{equation*}

If $Z$ is a closed subset of $\V$, we denote by $\cor_Z$ the sheaf associated to the closed subset $Z$. The canonical map $\cor_{\gamma^a} \to \cor_{\lbrace 0 \rbrace}$ induces a morphism
\begin{equation}\label{mor:gammacar}
F \sconv[np] \cor_{\gamma^a} \to F.
\end{equation}

\begin{proposition}[{\cite[Proposition 3.9]{GS2014}}]\label{prop:gammaloc}
Let $F \in \Derb(\cor_\V)$. Then $F \in \Derb_{\gamma^{\circ,a}}(\cor_\V)$ if and only if the morphism \eqref{mor:gammacar} is an isomorphism.
\end{proposition}
We finally recall the following notion extracted from \cite{Kashiwara2018a}.
\begin{definition}
Let $A$ be a subset of $\V$. We say that $A$ is $\gamma$-proper if  the map $s$ is proper on $\gamma \times A$.
\end{definition}

\subsubsection{Alexandrov topology}
Let $(X, \leq)$ be a preordered set. A lower (resp. upper) set $U$ of $(X, \leq)$ is a subset of $X$ such that if $x \in U$ and $y \in X$ with $y \leq x$ (resp. $x \leq y$) then $y \in U$.

By convention, the Alexandrov topology on $(X, \leq)$ is the topology whose open sets are the lower sets. A basis of this topology is given by the sets $\mathscr{D}(x)=\lbrace y \in X \vert \, y \leq x \rbrace$ for $x \in X$. Note that $\mathscr{D}(x)$ is the smallest open set containing $x$. We write $X_{\mathfrak{a} (\leq)}$ for $X$ endowed with the Alexandrov topology associated with the preorder $\leq$. If there is no risk of confusion, we omit the preoder and simply write $X_{\mathfrak{a}}$.

We recall the following classical fact.

\begin{proposition}
Let $(X, \leq)$ and $(Y,\preceq)$ be two preorders. A function $f \colon X_\mathfrak{a} \to Y_\mathfrak{a}$ is continuous if and only if $f\colon (X, \leq) \to (Y,\preceq)$ is order preserving.
\end{proposition}

\subsubsection{Alexandrov sheaves}
Let $\gamma$ be a closed proper convex cone of $\V$. The datum of  $\gamma$ endows $\V$ with the order 
\begin{center}
$x \leq_\gamma y$ if and only if $x+ \gamma \subset y+\gamma$. 
\end{center}
Consider the topological space $\V_{\al(\leq_\gamma)}$. For brevity, we write $\V_{\al(\gamma)}$ instead of $\V_{\al(\leq_\gamma)}$. If there is no risk of confusion, we write $\V_{\mathfrak{a}}$ instead of $\V_{\mathfrak{a}(\gamma)}$. An Alexandrov sheaf is an object of the abelian category $\Mod(\cor_{\V_{\al}})$. Recall that we denote by $\Der(\cor_{\V_{\al}})$ its derived category and by $\Derb(\cor_{\V_{\al}})$ its bounded derived category.

We denote by $\V_{\leq_\gamma}$ the category whose objects are the elements of $\V$ and given $x$ and $y$ in $\V$, there is exactly one morphism from $x$ to $y$ if an only if $x \leq_\gamma y$ . 
If there is no risk of confusion, we simply write 
$\V_\leq$ and set
\begin{equation*}
\Mod(\V_\leq) := \Fun((\V_\leq)^{\textnormal{op}},\Mod(\cor)).
\end{equation*}
A persistence module over $\V_\leq$ is an object of $\Mod(\V_\leq)$.
We write $\V^{\textnormal{top}}_\leq$ for $\V_\leq$ endowed with the trivial Grothendieck topology (that is the one for which all the sieves are representable). Note that on $\V^{\textnormal{top}}_\leq$ all presheaves are sheaves. Hence, the forgetful functor
$for \colon \V^{\textnormal{top}}_\leq \to \V_\leq$ induces an equivalence
\begin{equation*}
\Mod(\cor_{\V^{\textnormal{top}}_\leq})\stackrel{\sim}{\to}\Mod(\V_\leq).
\end{equation*}
For this reason, we will not distinguished between $\V^{\textnormal{top}}_\leq$ and $\V_\leq$.
There is a morphism of sites $\theta \colon \V_\al \to \V_\leq$ defined by
\begin{equation*}
 \theta^t \colon \V_\leq \to \Op(\V_\al), \quad x \mapsto x+\gamma.
\end{equation*}
The following statement is due to J. Curry. We refer to \cite{Kashiwara2018a} for a proof.

\begin{proposition} The functor
\begin{equation*}
\theta_\ast \colon \Mod(\cor_{\V_{\al}})  \to \Mod(\V_\leq)
\end{equation*}
is an equivalence of categories.
\end{proposition}

\subsection{Relation between $\gamma$-sheaves and Alexandrov sheaves}

In order to compare $\gamma$-sheaves and Alexandrov sheaves we use morphisms of sites. These are morphisms between Grothendieck topologies and in particular usual topologies considered as Grothendieck topologies. It is important to keep in mind that some morphisms of sites between usual topological spaces are not induced by continuous maps. This is why we use this notion. Operations on sheaves can also be defined for morphisms of sites. These operations on sheaves generalize the ones induced by continuous maps between topological spaces. We refer the reader to \cite{KS2006} for a detailed presentation.

Let $\V$ be a finite dimensional real vector space and $\gamma$ a cone of $\V$ satisfying hypothesis \eqref{hyp:cone}. Recall that we have defined a preorder $\leq$ on $\V$ as follow : 

\begin{equation*}
x \leq_\gamma y \Leftrightarrow x+\gamma \subset y +\gamma.
\end{equation*}
By definition of the Alexandrov topology, the open sets $\mathscr{D}(x)=x + \gamma$ for $x \in \V$ form a base of the topology of $\V_\al$. 

We define the functor $\alpha^t \colon \Op(\V_\al) \to \Op(\V_\gamma)$ by
\begin{equation*}
\alpha^t \colon \Op(\V_\al) \to \Op(\V_\gamma), \quad U=\bigcup_{x \in U} x+\gamma \mapsto \bigcup_{x \in U} x+\Int(\gamma).
\end{equation*}

\begin{lemma}
The functor $\alpha^t$ is a morphism of sites $\alpha \colon \V_\gamma \to \V_\al$.
\end{lemma}

\begin{proof}
It is clear that $\alpha^t$ preserves covering. Let us check that it preserves finite limits. For that purpose it is sufficient to check that it preserves the final object (clear) and fibered products which reduces in this particular setting to show that
\begin{equation*}
\alpha^t(U \cap  V)=\alpha^t(U) \cap \alpha^t(V).
\end{equation*}
On one hand
\begin{align*}
\alpha^t(U \cap  V)&= \bigcup_{x \in U \cap V} x+\Int(\gamma). 
\end{align*}
On the other hand
\begin{align*}
 \alpha^t(U) \cap  \alpha^t(V)&= \!\!\bigcup_{z \in \alpha^t(U) \cap  \alpha^t(V)} \!\!z+\Int(\gamma).
\end{align*}
Hence, $ \alpha^t(U) \cap \alpha^t(V) \subset \alpha^t(U \cap  V) $. As $U \cap V$ is included in $U$ and $V$ it follows by functoriality that $\alpha^t(U \cap V)$ is included in $\alpha^t(U)$ and $\alpha^t(V)$. This proves the reverse inclusion
\end{proof}

We also have the following morphism of sites
\begin{equation*}
\beta \colon \V_\al \to \V_\gamma ,\quad \beta^t(x+\Int(\gamma))=x+\Int(\gamma).
\end{equation*}

\begin{fact} \label{fact:essential}
The composition of $\beta$ and $\alpha$ satisfies $\beta \circ \alpha = \id$. 
\end{fact}

The morphism of sites $\alpha$ and $\beta$ provide the following adjunctions
\begin{align*}
\xymatrix{
\alpha^{-1} \colon \Mod(\cor_{\V_\al}) \ar@<.4ex>[r]& \Mod(\cor_{\V_\gamma}) \ar@<.4ex>[l] \colon \alpha_\ast
,}
\\
\xymatrix{
 \beta^{-1} \colon \Mod(\cor_{\V_\gamma}) \ar@<.4ex>[r]& \Mod(\cor_{\V_\al})  \ar@<.4ex>[l] \colon \beta_\ast.}
\end{align*}

We define the functor
\begin{equation*}
\alpha^\dagger \colon \Fun(\Op(\V_\al)^{\textnormal{op}},\Mod(\cor)) \to \Fun(\Op(\V_\gamma)^{\textnormal{op}},\Mod(\cor)), \quad F \mapsto \alpha^\dagger F
\end{equation*}
where
\begin{equation*}
\textnormal{for every $U \in \Op(\V_\gamma)$}, \; \alpha^\dagger F(U)= \underset{U \subset \alpha^t(V)}{\colim} F(V).
\end{equation*}
Recall that by definition $\alpha^{-1}F$ is the sheafification of $\alpha^\dagger F$.

\begin{proposition}\label{prop:adjunctpro}
\begin{enumerate}[(i)]
\item There are canonical isomorphisms of functors $\alpha^{-1} \simeq \alpha^\dagger \simeq \beta_\ast$,
\item the functor $\alpha_\ast$ is fully faithful,
\item the functor $\beta^{-1}$ is fully faithful.
\end{enumerate}
\end{proposition}

\begin{proof}
\noindent(i)
Let $F \in \Mod(\cor_{\V_\al})$. Then,
\begin{align*}
\alpha^\dagger F(U)= \underset{U \subset \alpha^t(V)}{\colim} F(V)=F(U)=\beta_\ast F(U).
\end{align*}
Hence, $\alpha^\dagger \simeq \beta_\ast$. Since $\alpha^\dagger F$ is a sheaf, it follows that $\alpha^{-1} \simeq \alpha^\dagger$.\\

\noindent(ii)  Let $F, \; G \in \Mod(\cor_{\V_\gamma})$. The isomorphism of functors $\beta_\ast \alpha_\ast \simeq \id$ implies that the morphism $\Hom_{\cor_{\V_\gamma}}(F,G) \stackrel{\alpha_\ast}{\longrightarrow}  \Hom_{\cor_{\V_\al}}(\alpha_\ast F,\alpha_\ast G)$ is injective.
Let $\phi \in \Hom_{\cor_{\V_\al}}(\alpha_\ast F,\alpha_\ast G)$. Set $\psi_{x + \Int(\gamma)}:= \phi_{x + \gamma}$. Since $\lbrace x + \Int(\gamma) \rbrace_{x \in \V_\gamma}$ is a basis of $\V_\gamma$, the family $(\psi_{x + \Int(\gamma)})_{x \in \V_\gamma}$ defines a morphism of sheaves $\psi \colon F \to G$ and  $\alpha_\ast \psi =\phi$. This proves that $\alpha_\ast$ is fully faithful.\\

\noindent(iii) This follows from \cite[Exercise 1.14]{KS2006}.
\end{proof}
We have the following sequence of adjunctions $\beta^{-1} \dashv \beta_\ast \simeq \alpha^{-1} \dashv \alpha_\ast$.

\begin{example}
The functors $\alpha_\ast$ and $\beta^{-1}$ are different as the following example shows. We set $\V=\R$ and $\gamma=]-\infty, 0]$. We consider the $\gamma$-closed set $[t, +\infty[$ with $t \in \R$ and the sheaf $\cor_{[t, +\infty[}$ associated with it. Consider the sheaves
\begin{equation*}
\beta^{-1}\cor_{[t, +\infty[} \quad \textnormal{and} \quad \alpha_\ast \cor_{[t, +\infty[}.
\end{equation*}

We compute the stalk at $t$ of these two sheaves. For the first one, observe that  the continuous map $\beta \colon\V_\al \longrightarrow \V_\gamma$ is the identity on the elements of $\V$. Therefore, we have $(\beta^{-1}\cor_{[t, +\infty[})_t \simeq (\cor_{[t, +\infty[})_t \simeq \cor$. For the second one,
\begin{align*}
(\alpha_\ast\cor_{[t, +\infty[})_t & \simeq \alpha_\ast\cor_{[t, +\infty[}(t+\gamma)\\
                                   & \simeq \cor_{[t, +\infty[}(]-\infty,t[)\\
                                   &\simeq 0.                                   
\end{align*}
\end{example}

\subsection{Compatibilities of operations}
In this subsection, we study the compatibility between operations for sheaves in $\gamma$ and Alexandrov topologies.

Let $(\V,\gamma)$ and $(\W,\lambda)$ be two finite dimensional real vector spaces endowed with cones satisfying the hypothesis \eqref{hyp:cone}. 

\begin{lemma}
Let $f \colon \V \to \W$ be a linear map. The following statements are equivalent.
\begin{enumerate}[(i)]
\item $f(\gamma) \subset \lambda$,
\item $f \colon \V_\gamma \to \W_\lambda$ is continuous,
\item $f \colon \V_{\al(\gamma)} \to \W_{\al(\lambda)}$ is continuous.
\end{enumerate}
\end{lemma}
 
\begin{proof}
\noindent (i)$\Rightarrow$(ii) Let $y \in \W$. Let us show that $f^{-1}(y+\Int(\lambda))$ is a $\gamma$-open. As $\V$ and $\W$ are finite dimensional, $f$ is continuous for the usual topology. Hence $f^{-1}(y+\Int(\lambda))$ is open. The inclusion $f^{-1}(y+\Int(\lambda)) \subset f^{-1}(y+\Int(\lambda)) + \gamma$ is clear. Let us show the reverse inclusion.
Let $x \in f^{-1}(y+\Int(\lambda)) + \gamma$. There exists $u \in f^{-1}(y+\Int(\lambda))$ and $v \in \gamma$ such that $x=u+v$. Then 
$f(x)=y+l+f(v)$ with $l \in \Int(\lambda)$ and $f(v) \in \lambda$. Since $\Int(\lambda)=\Int(\lambda)+\lambda$ it follows that $f(x) \in y+\Int(\lambda)$. Hence $f^{-1}(y+\Int(\lambda)) + \gamma=f^{-1}(y+\Int(\lambda))$. This proves that $f^{-1}(y+\Int(\lambda))$ is a $\gamma$-open.

\noindent(ii)$\Rightarrow$(i) Since $f(0)=0$ and $f$ is continuous,  for every $\varepsilon > 0$ there exists $\eta >0$ such that $f(B(0;\eta)+\gamma) \subset B(0,\varepsilon)+\lambda$. Hence, if $v \in \gamma$, $f(v) \in \overline{\lambda}=\lambda$.

\noindent (i)$\Rightarrow$(iii) The statement (i) implies that $f \colon (\V, \leq_\gamma) \to (\W, \leq_\lambda)$ is order preserving. Hence, $f \colon \V_{\al(\gamma)} \to \W_{\al(\lambda)}$ is continuous.

\noindent (iii)$\Rightarrow$(i) $\lambda$ is an open subset of $\W_{\al(\lambda)}$. As $f^{-1}(\lambda)$ is an open subset of $\V_{\al(\gamma)}$ such that $0 \in f^{-1}(\lambda)$ it follows that $\gamma \subset f^{-1}(\lambda)$. Hence $f(\gamma) \subset \lambda$.
\end{proof}
Let $f \colon \V \to \W$ be a linear map. Assume that $f(\gamma) \subset \lambda$. We denote by $f_\al \colon \V_{\al(\gamma)} \to \W_{\al(\lambda)}$ the continuous map between $\V_{\al(\gamma)}$ and $\W_{\al(\lambda)}$ whose underlying linear map is $f$.

\begin{proposition}
\begin{enumerate}[(i)]
\item Assume that $f \colon \V_\gamma \to \W_\lambda$ is continuous. Then the following diagram of morphisms of sites is commutative.
\begin{equation*}
\xymatrix{
 \V_\al \ar[r]^-{\beta} \ar[d]_-{f_\al} & \V_\gamma \ar[d]^-{f} \\
 \W_\al \ar[r]^-{\beta} & \W_\lambda
}
\end{equation*}

\item Assume that $f(\Int(\gamma)) \subset \Int(\lambda)$. Then the following diagram of morphisms of sites is commutative.
\begin{equation*}
\xymatrix{
\V_\gamma \ar[r]^-{\alpha} \ar[d]_-{f}& \V_\al \ar[d]^-{f_\al}  \\
\W_\lambda \ar[r]\ar[r]^-{\alpha} & \W_\al
}
\end{equation*}
\end{enumerate}
\end{proposition}

\begin{proof}
\noindent (i) is clear.\\
\noindent(ii) 
Let $y \in \W$. On one hand, we have
\begin{align*}
\alpha^t \circ f_\al^t(y+\lambda)&=\alpha^t(\bigcup_{ \lbrace x \in \V \vert f(x) \in y+\lambda \rbrace} x+ \gamma)\\
&=\bigcup_{ \lbrace x \in \V \vert f(x) \in y+\lambda \rbrace} x+ \Int(\gamma).
\end{align*}
On the other hand,
\begin{align*}
f^t \circ \alpha^t (y+\lambda)&=f^t(y+\Int(\lambda))\\
&=\bigcup_{ \lbrace x \in \V \vert f(x) \in y+\Int(\lambda) \rbrace} x+ \Int(\gamma).
\end{align*}
The inclusion 
\begin{align*}
\bigcup_{ \lbrace x \in \V \vert f(x) \in y+\Int(\lambda) \rbrace} \hspace{-0.5cm}x+ \Int(\gamma)\;\subset \bigcup_{ \lbrace x \in \V \vert f(x) \in y+\lambda \rbrace} x+ \Int(\gamma)
\end{align*} 
is clear. Let us prove the reverse inclusion. Let $z \in \bigcup_{ \lbrace x \in \V \vert f(x) \in y+\lambda \rbrace} x+ \Int(\gamma)$. Then $z=x+g$ with $g \in \Int(\gamma)$ and $f(z)=y+l+f(g)$ with $ l \in \lambda$. As $f(g) \in \Int(\lambda)$ then $l+f(g) \in \Int(\lambda)$. It follows that $f(z) \in y+ \Int(\lambda)$.   
\end{proof}

\begin{example} In $(ii)$ the hypothesis $f(\Int(\gamma)) \subset \Int(\lambda)$ is necessary as shown in the following example.

On $\R$, consider the cone $\gamma=\lbrace x \in \R \vert \,x \leq 0 \rbrace$ and on $\R^2$ consider the cone $\lambda=\lbrace (x,y) \in \R^2 \vert \, x \leq 0 \; \textnormal{and} \; y \leq 0 \rbrace$.
Let $f \colon \R \to \R^2$, $x \mapsto (x,0)$. Then, computing both $f^t\alpha^t(\lambda)$ and $\alpha^t f_\al^t(\lambda)$, we get

\begin{tabular}{p{7cm}p{7cm}}
{\begin{align*}
 f^t\alpha^t(\lambda) &= f^t(\Int(\lambda))\\
&= \emptyset
\end{align*}}
&
{\begin{align*}
\alpha^t f_\al^t (\lambda) &= \alpha^t(\gamma)\\
&= \Int(\gamma)\\
\end{align*}}
\end{tabular}

\noindent Note that the condition $f(\Int(\gamma)) \subset \Int(\lambda)$ is automatically satisfied when $f$ is surjective.
\end{example}

\section{Ephemeral persistent modules}
\subsection{The category of ephemeral modules}
In this section, we propose a notion of ephemeral persistent module in arbitrary dimension, generalizing the one of \cite{Chazal2016a}. For the convenience of the reader, we start by recalling the definition of a Serre subcategory and of the quotient of an abelian category by a Serre subcategory that we subsequently use. We refer the reader to \cite{Gabriel1962} and \cite[\href{https://stacks.math.columbia.edu/tag/02MN}{Tag 02MN}]{stacks-project}.

\begin{definition}
Let $\cA$ be an abelian category. A Serre subcategory $\cC$ of $\cA$ is a non-empty full subcategory of $\cA$ such that given an exact sequence 
\begin{equation*}
X \to A \to Y 
\end{equation*}
with $X$ and $Y$ in $\cC$ and $A \in \cA$ then $A \in \cC$. 

If $\cC$ is closed under isomorphism, we say that it is a strict Serre subcategory of $\cA$.
\end{definition}

\begin{lemma}
Let $\cA$ be an abelian category and $\cC$ be a Serre subcategory of $\cA$. There exists an abelian category denoted $\cA / \cC$ and an exact functor $Q \colon \cA \to \cA / \cC$ whose kernel is $\cC$ satisfying the following universal property: For any exact functor $F \colon \cA \to \cB$ such that $\cC \subset \Ker(F)$ there exists a factorization $F=H \circ Q$ for a unique exact functor $H \colon \cA / \cC \to \cB$.
\end{lemma}

\begin{proposition}[{\cite[Ch.2 \S 2 Proposition 5]{Gabriel1962}}]\label{prop:comploc}
Let $L \colon \cA \to \cB$ be an exact functor between abelian categories. Assume that $L$ has a fully faithful right adjoint $R$. Then $\Ker L$ is a Serre subcategory of $\cA$ and $L$ induces an equivalence between $\cA/ \Ker(L)$ and $\cB$.
\end{proposition}
We now introduce our notion of ephemeral module.
\begin{definition}\label{def:eph}
An object $G \in \Mod(\cor_{\V_\al})$ is ephemeral if and only if $\beta_\ast G \simeq 0$. We denote by $\Eph(\cor_{\V_\al})$ the full subcategory of $\Mod(\cor_{\V_\al})$ spanned by ephemeral modules.
\end{definition}

In other words, an object $G \in \Mod(\cor_{\V_\al})$ is ephemeral if and only if for every open subset $U$ of the usual topology of $\V$, $G(U+\gamma)=0$.

\begin{lemma}\label{lem:Serre}
The full subcategory $\Eph(\cor_{\V_\al})$ of $\Mod(\cor_{\V_\al})$ is a strict Serre subcategory, stable by limits and colimits.
\end{lemma}

\begin{proof}
Since $\beta_\ast\simeq \alpha^{-1}$, $\Eph(\cor_{\V_\al}) \simeq \ker(\alpha^{-1})$. Since $\alpha^{-1}$ is exact, $\Eph$ is a Serre subcategory of $\Mod(\cor_{\V_\al})$. Since $\beta_\ast$ commutes with limits (it is a right adjoint) and $\alpha^{-1}$ commutes with colimits (it is a left adjoint), $\Eph(\cor_{\V_\al})$ has limits and colimits.
\end{proof}

\begin{theorem}\label{thm:GammaSh_SerreQ}
The functor $\alpha^{-1} \colon \Mod(\cor_{\V_\al}) \to \Mod(\cor_{\V_\gamma})  $ induces an equivalence of categories between $\Mod(\cor_{\V_\al}) / \Eph(\cor_{\V_\al})$ and $\Mod(\cor_{\V_\gamma})$.
\end{theorem}

\begin{proof}
This is a direct consequences of Proposition \ref{prop:adjunctpro} and \ref{prop:comploc}.
\end{proof}

\subsection{Ephemeral modules on $\R$}

Ephemeral modules on $(\R, \leq)$ where introduced in \cite{Chazal2016a} and the category of observable modules on $\R$ was introduced and studied in \cite{Chazal2016b}. We show that our notion of ephemeral module generalize to arbitrary dimension the one of \cite{Chazal2016a} and \cite{Chazal2016b}.

The convention of \cite{Chazal2016b} are equivalent in our setting to the choice of the proper closed convex cone $\gamma=[0,+\infty[$.

\begin{lemma} \label{lem:compeph} Let $F \in \Mod(\cor_{\R_{\al(\gamma)}})$. The following are equivalent,
\begin{enumerate}[(i)]
\item $F \in \Eph(\cor_{\R_{\al(\gamma)}})$,
\item  the restriction morphism $\rho_{t,s} \colon F(s+\gamma) \to F(t+\gamma)$ is null whenever $s < t$.
\end{enumerate}
\end{lemma}

\begin{proof}
\noindent (i)$\Rightarrow$(ii). There exists $u \in \R$ such that $s<u<t$ and by hypothesis $F(u+\Int(\gamma))\simeq (0)$. Hence, we have the following commutative diagram
\begin{equation*}
\xymatrix{ F(s+\gamma) \ar[dr]_-{\rho_{u+\Int(\gamma),s}} \ar[rr]^{\rho_{t,s}}&& F(t + \gamma)\\
& (0). \ar[ur]_-{\rho_{t,u+\Int(\gamma)}}&
}
\end{equation*}
This implies that $\rho_{t,s}=0$.\\

\noindent (ii)$\Rightarrow$(i). As the family $(x+\Int(\gamma))_{x \in \R}$ is a basis of the $\gamma$-topology on $\R$, it is sufficient to show that for every $x \in \R$, $F(x+\Int(\gamma))=(0)$. Let $x \in \R$. Since $F$ is a sheaf for the Alexandrov topology, we have the following isomorphism
\begin{equation}\label{map:isonull}
\lim_{u+\gamma \subset x+\Int(\gamma)} \rho_{x+ \Int(\gamma),u} \colon F(x+\Int(\gamma)) \stackrel{\sim}{\to} \lim_{u+\gamma \subset x+\Int(\gamma)}F(u+\gamma).
\end{equation}
Since $u+\gamma \subset x+ \Int(\gamma)$, $x<u$. Then, there exists $t \in \R$ such that $x<t<u$. Hence, $\rho_{x+ \Int(\gamma),u}=\rho_{t,u} \circ \rho_{x+ \Int(\gamma),t}=0$. It follows that the isomorphism \eqref{map:isonull} is the zero map. This implies that $F(x+\Int(\gamma))\simeq 0$.
\end{proof}

We refer the reader to \cite[Definition 2.3]{Chazal2016b} for the definition of the observable category denoted $\textnormal{\textbf{Ob}}$ and recall the following result by the same authors

\begin{theorem}[{\cite[Corollary 2.13]{Chazal2016b}}]\label{thm:Obscat}
There is the following equivalence of categories $\textnormal{\textbf{Ob}} \simeq \Mod(\cor_{\R_{\al(\gamma)}})/\Eph(\cor_{\R_{\al(\gamma)}})$.
\end{theorem}

A special case of the following result already appears in \cite[Corollary 6.7]{Berkouk}.

\begin{corollary}\label{cor:obgam}
The observable category $\textnormal{\textbf{Ob}}$ is equivalent to the category $ \Mod(\cor_{\R_{\gamma}})$.
\end{corollary}

\begin{proof}
Using Theorem \ref{thm:GammaSh_SerreQ} and Theorem \ref{thm:Obscat}, we obtain the following sequence of equivalence
\begin{equation*}
\textnormal{\textbf{Ob}} \simeq \Mod(\cor_{\R_{\al(\gamma)}})/\Eph(\cor_{\R_{\al(\gamma)}}) \simeq  \Mod(\cor_{\R_{\gamma}}).
\end{equation*}
\end{proof}

\subsection{Ephemeral modules in the derived category}

We write $\Der(\cor_{V_\al})$ for the derived category of Alexandrov sheaves and $\Der(\cor_{V_\gamma})$ for the one of $\gamma$-sheaves.

It follows from the preceding subsections that we have the following adjunctions
\begin{align*}
\xymatrix{
\beta_\ast=\alpha^{-1} \colon \Der(\cor_{\V_\al}) \ar[r] & \Der(\cor_{\V_\gamma}) \ar@<.8ex>[l] \ar@<-.8ex>[l] \colon \opb{\beta}, \; \roim{\alpha}.\\
}
\end{align*}

\begin{proposition}
\begin{enumerate}[(i)]
\item the functor $\beta^{-1}$ is fully faithful,
\item the functor $\roim{\alpha}$ is fully faithful.
\end{enumerate}
\end{proposition}
\begin{proof}
\noindent (i) follows from Proposition \ref{prop:adjunctpro} as $\beta_\ast$ and $\opb{\beta}$ are exact.\\
\noindent(ii) This follows from \cite[Exercices 1.14]{KS2006}.
\end{proof}

\begin{proposition}\label{prop:cohbound}
The functor $\oim{\alpha}$ has finite cohomological dimension.
\end{proposition}

\begin{proof}
Let $F \in \Mod(\cor_{\V_\gamma})$. Since $\V$ is a real vector space of dimension $n$, it follows that there exists an injective resolution $\cI$ of $\opb{\phi_\gamma} F$ of the form
\begin{equation*}
0 \to I^0 \to I^1 \to \cdots \to I^{n+1} \to 0.
\end{equation*}
As $\roim{\phi_\gamma}\opb{\phi_\gamma} F \simeq F$, it follows that $F \simeq   \oim{\phi_\gamma} \cI$. Since $\oim{\phi_\gamma}$ preserves bounded complexes of injectives, $\oim{\phi_\gamma} \cI$ is again a bounded complex of injectives. Thus,
\begin{equation*}
\roim{\alpha} F \simeq \oim{\alpha}\oim{\phi_\gamma} \cI. 
\end{equation*}
Hence, for $k > n+1$, 
\begin{equation*}
\Hn^k(\roim{\alpha} F) \simeq \Hn^k(\oim{\alpha} \oim{\phi_\gamma} \cI) \simeq 0.
\end{equation*}
\end{proof}

\begin{remark}
It follows from Proposition \ref{prop:cohbound} that the functor $\roim{\alpha} \colon \Derb(\cor_{\V_\gamma}) \to \Derb(\cor_{\V_\al})$ is well defined. Note that, in this paper, the results stated for the unbounded derived categories $\Der(\cor_{\V_\gamma})$ and $\Der(\cor_{\V_\al})$ also hold for their bounded counterparts $\Derb(\cor_{\V_\gamma})$ and $\Derb(\cor_{\V_\al})$.
\end{remark}

We write $\Der_\Eph(\cor_{\V_\al})$ for the full subcategory of $\Der(\cor_{\V_\al})$ consisting of objects $F \in \Der(\cor_{\V_\al})$ such that for every $i \in \Z$, $\Hn^i(F) \in \Eph(\cor_{\V_\al})$. Since $\Eph(\cor_{\V_\al})$ is a thick abelian subcategory of $\Mod(\cor_{\V_\al})$, $\Der_\Eph(\cor_{\V_\al})$ is a triangulated subcategory of $\Der(\cor_{\V_\al})$. We consider the full subcategory of $\Der(\cor_{\V_\al})$
\begin{equation*}
\Ker{\opb{\alpha}}= \lbrace F \in \Der(\cor_{\V_\al}) \; \vert \; \opb{\alpha} F \simeq 0 \rbrace.
\end{equation*}

Recall that a subcategory $\cC$ of a triangulated category $\cT$ is thick if it is triangulated and it contains all direct summands of its objects. It is clear that $\Ker{\opb{\alpha}}$ is thick and closed by isomorphisms.

\begin{lemma}
The triangulated category $\Der_\Eph(\cor_{\V_\al})$ is equivalent to the triangulated category $\Ker \opb{\alpha}$.
\end{lemma}

\begin{proof}
This follows immediately form the exactness of $\alpha^{-1}$.
\end{proof}

We now briefly review the notion of localization of triangulated categories. References are made to \cite[Chapter 7]{KS2006} and \cite{Krausea}.

Let $\cT$ be a triangulated category and $\cN$ be a triangulated full subcategory of $\cT$. We write $W(\cN)$ for the set of maps $f:X \to Y$ of $\cT$ which sit into a triangle of the form
\begin{equation*}
X \stackrel{f}{\to} Y \to Z \stackrel{+1}{\to}
\end{equation*}
where $Z \in \cN$. By definition the quotient of $\cT$ by $\cN$ is the localization of $\cT$ with respect to the set of maps $W(\cN)$. That is
\begin{equation*}
\cT /\cN := \cT[W(\cN)^{-1}]
\end{equation*} 
together with the localization functor
\begin{equation*}
Q \colon \cT \to \cT /\cN.
\end{equation*}
The following proposition is well-known.
\begin{proposition}
Let $L \colon \cC \leftrightarrows \cD \colon R$ be an adjunction. Assume that the right adjoint $R$ is fully faithful. Then $L \colon \cC \to \cD$ is the localization of $\cC$ with respect to the set of morphisms 
\begin{equation*}
W=\lbrace f : X \to Y \in \Mor(\cC) \; \vert \; L(f) \;\textnormal{is an isomorphism} \rbrace.
\end{equation*}
\end{proposition}

\begin{proposition}
The category $\Der(\cor_{\V_\gamma})$ is the quotient of $\Der(\cor_{\V_\al})$ by $\Der_\Eph(\cor_{\V_\al})$ via the localization functor $\alpha^{-1} \colon \Der(\cor_{\V_\al}) \to \Der(\cor_{\V_\gamma})$. In particular, $\Der(\cor_{\V_\al})/\Der_\Eph(\cor_{\V_\al}) \simeq \Der(\cor_{\V_\gamma})$.
\end{proposition}

\begin{proof}
Let $W=\lbrace f  \in \Mor(\cC) \; \vert \; \opb{\alpha}(f) \;\textnormal{is an isomorphism} \rbrace$.  Let $f\colon F \to G$ be a morphism of $\Der(\cor_{\V_\al})$.
By the axiom of triangulated categories, $f$ sits in a distinguished triangle
\begin{equation*}
F \stackrel{f}{\to} G \to H \stackrel{+1}{\to}.
\end{equation*}
Hence $\opb{\alpha}(f)$ is an isomorphism if and only if $\opb{\alpha} H \simeq 0$. That is if $H \in \Der_{\Eph}(\cor_{\V_\al})$. This proves the claim.
\end{proof}

\section{Distances on categories of sheaves}
\subsection{Preliminary facts}
Let $\V$ be a finite dimensional vector space, $\gamma \subset \V$ be a cone satisfying \eqref{hyp:cone} and $v\in \V$. The map
\begin{equation*}
\tau_v \colon \V \to \V, \; x \mapsto x - v
\end{equation*}
is continuous for the usual, the Alexandrov and the $\gamma$ topologies on $\V$.

\subsubsection{Alexandrov \& $\gamma$-topology}

Let $v, \;w \in \V$ and assume that $w \leq_\gamma v$. Let $F \in \Mod(\cor_{\V_\al})$ (resp. $\Mod(\cor_{\V_\gamma})$). Since $w+\gamma \subset v+\gamma$, it follows that for every $U \in \Op(\V_\al)$ (resp. $\Op(\V_\gamma)$), $U +w \subset U+v$. Hence, the restriction morphisms $\rho_{U+v,U+w}$ of $F$ allows to define a morphism of sheaves
\begin{equation*}
\chi^\al_{v,w}(F) \colon \oim{\tau_v} F \to \oim{\tau_w} F
\end{equation*}
by setting for every open subset $U$, $\chi^\al_{v,w}(F)_U:=\rho_{U+v,U+w}$.

This construction is extended to the derived category $\Der(\cor_{\V_\al})$ as follows. 
Let $F \in \Der(\cor_{\V_\al})$. Replacing $F$ by an homotopically injective resolution $\cI$, and using the restriction morphisms of $\cI$ as in the preceeding construction, we obtain a morphism of sheaves
\begin{equation*}
\oim{\tau_v} \cI \to \oim{\tau_w} \cI.
\end{equation*}
This provides a morphism
\begin{equation*}
\chi^\al_{v,w}(F) \colon \oim{\tau_v} F \to \oim{\tau_w} F.
\end{equation*}
It follows that there is a morphism of functors from $\Der(\cor_{\V_\al})$ to $\Der(\cor_{\V_\al})$ 
\begin{equation}\label{mor:bismooth}
\chi^\al_{v,w} \colon \oim{\tau_v} \to \oim{\tau_w}.
\end{equation}
In a similar way, we obtain a morphism of functors from $\Der(\cor_{\V_\gamma})$ to $\Der(\cor_{\V_\gamma})$
\begin{equation}\label{mor:bismoothgamma}
\chi^\gamma_{v,w} \colon \oim{\tau_v} \to \oim{\tau_w}.
\end{equation}
One immediately verify that for every $F \in \Der(\cor_{\V_\al})$ and $G \in \Der(\cor_{\V_\gamma})$
\begin{equation}\label{mor:smoothing_op_betastar}
\beta_\ast \chi^\al_{v,w}(F) \simeq \chi^\gamma_{v,w} \,(\oim{\beta} F),
\end{equation}
\begin{equation}\label{mor:smoothing_op_alphastar}
\roim{\alpha} \chi^\gamma_{v,w}(G) \simeq \chi^\al_{v,w} \,(\roim{\alpha} G).
\end{equation}

\begin{lemma}\label{lem:beta_smooth}
For every $F \in \Der(\cor_{\V_\gamma})$, there is the following canonical isomorphism
\begin{equation*}
\opb{\beta} \chi^\gamma_{v,w}(F) \simeq \chi^\al_{v,w} \,(\opb{\beta}F).
\end{equation*}
\end{lemma}

\begin{proof}
Let $F \in \Der(\cor_{\V_\gamma})$ and consider the canonical morphism.
\begin{equation*}
\chi^\al_{v,w}(\opb{\beta}F)\colon \oim{\tau_v}\opb{\beta}F \to \oim{\tau_w}\opb{\beta}F.
\end{equation*}
Since $\opb{\beta}$ is fully faithful and commutes with $\oim{\tau_v}$ and $\oim{\tau_w}$, there exists a unique morphism $f\colon \oim{\tau_v}F \to \oim{\tau_w}F$ such that the following diagram commutes
\begin{equation*}
\xymatrix@C=1.6cm{
 \oim{\tau_v}\opb{\beta}F \ar[r]^-{\chi^\al_{v,w}(\opb{\beta}F)} \ar[d]_-{\wr} 
 & \oim{\tau_w}\opb{\beta}F \ar[d]^-{\wr}\\
\opb{\beta}\oim{\tau_v}F \ar[r]^-{\opb{\beta}f}& \opb{\beta}\oim{\tau_w}F.
}
\end{equation*}
Hence, $\chi^\al_{v,w}(\opb{\beta}f) \simeq \opb{\beta}f$. Applying $\oim{\beta}$ to the preceding formula, we get $\oim{\beta}\chi^\al_{v,w}(\opb{\beta} F) \simeq \oim{\beta}\opb{\beta}f$. It follows from the fully faithfulness of $\opb{\beta}$ and from Formula \eqref{mor:smoothing_op_betastar} that
\begin{equation*}
\chi^\gamma_{v,w}(F) \simeq f.
\end{equation*}
Thus, $\opb{\beta}\chi^{\gamma}_{v,w}(F) \simeq \chi^\al_{v,w}(\opb{\beta} F)$.

\end{proof}

Let $F \in \Der(\cor_{\V_\al})$ and  $G \in \Der(\cor_{\V_\gamma})$. If $w=0$ and $v \in \gamma^a$, the morphisms \eqref{mor:bismooth} and \eqref{mor:bismoothgamma} provide respectively  the canonical morphisms

\begin{equation*}\label{mor:smooth}
\chi^\al_{v,0} \colon \oim{\tau_v}F \to F,
\end{equation*} 

\begin{equation*}\label{mor:smoothgamma}
\chi^\gamma_{v,0} \colon  \oim{\tau_v}G \to G.
\end{equation*}

\begin{remark}\label{Rem:ab}
In the abelian cases i.e. for the categories $\Mod(\cor_{\V_\al})$ and $\Mod(\cor_{\V_\gamma})$ similar morphisms exist. They can be constructed directly or induced from the derived cases by using the following facts. If $\cA$ is an abelian category and $\Der(\cA)$ is its derived category, then the canonical functor
\begin{align*}
 \iota \colon \cA \to \Der(\cA)  
\end{align*}
which send an object of $\cA$ to the corresponding complex concentrated in degree zero is fully faithful. Moreover, $\Hn^0 \circ \, \iota \simeq \id$ and for every $v \in V$, $\oim{\tau_v}$ is exact and thus, commutes with $\Hn^0$. Hence, we will focus on the derived situations as it implies, here, the abelian case. 
\end{remark}

\subsubsection{The microlocal setting}
We now construct similar morphisms for sheaves in $\Derb_{\gamma^{\circ,a}}(\cor_\V)$. This construction is classical (see for instance \cite{GS2014}). We provide it for the convenience of the reader.

\begin{lemma}\label{lem:transker}
Let $F \in \Derb_{\gamma^{\circ\,a}}(\cor_{\V})$ and $u \in \V$. Then there is a functorial isomorphism
\begin{align*}
\tau_{u \ast} F \simeq \cor_{-u+\gamma^a} \sconv[np] F.
\end{align*}
\end{lemma}

\begin{proof}
It follows from Proposition \ref{prop:gammaloc} that the canonical morphism $\cor_{\gamma^a} \sconv[np] F \to F$ is an isomorphism and $\tau_u \circ s =s \circ (\tau_u \times \id)$. Hence
\begin{align*}
\tau_{u \ast} F &\simeq \tau_{u \ast} (\cor_{\gamma^a} \sconv[np] F)\\ 
                & \simeq s_\ast (\tau_{u} \times \id)_\ast (\cor_{\gamma^a} \boxtimes F)\\
                & \simeq \cor_{-u+\gamma^a} \sconv[np] F. 
\end{align*} 
\end{proof}
For $w\leq_\gamma v$, the canonical map
\begin{equation*}
\cor_{-v+\gamma^a} \to \cor_{-w+\gamma^a} 
\end{equation*}
induces a morphism of functors
\begin{equation}
\cor_{-v+\gamma^a} \sconv[np] (\cdot) \to \cor_{-w+\gamma^a} \sconv[np] (\cdot).
\end{equation}
Using Lemma \ref{lem:transker}, we obtain a morphism of functors from $\Derb_{\gamma^{\circ\,a}}(\cor_{\V})$ to $\Derb_{\gamma^{\circ\,a}}(\cor_{\V})$ 

\begin{equation}\label{mor:bismoothmu}
\chi^\mu_{v,w} \colon \oim{\tau_v} \to \oim{\tau_w}.
\end{equation}

\begin{lemma}\label{lem:microgam} Let $F \in \Derb_{\gamma^{\circ\,a}}(\cor_{\V})$ and $G \in \Derb(\cor_{\V_\gamma})$. There are the following canonical isomorphisms
\begin{equation}\label{iso:gammamicro}
\roim{\phi_\gamma} \chi^\mu_{v,w}(F) \simeq \chi^\gamma_{v,w} (\roim{\phi_\gamma}F),
\end{equation}

\begin{equation}\label{iso:microgamma}
\opb{\phi_\gamma} \chi^\gamma_{v,w}(G) \simeq \chi^\mu_{v,w} (\opb{\phi_\gamma}G).
\end{equation}
\end{lemma}

\begin{proof}
Let $F \in \Derb_{\gamma^{\circ\,a}}(\cor_{\V})$ and $U$ be a $\gamma$-open set. Then we have the following commutative diagram

\begin{align*}
\xymatrix@C=4cm{
\RHom_{\cor_{\V_\gamma}}(\cor_U,\, \roim{\phi_\gamma} \oim{\tau_v} F) \ar[r]^-{\RHom_{\cor_{\V_\gamma}}(\cor_U,\roim{\phi_\gamma} \chi^\mu_{v,w})} \ar[d]^-{\wr} & \RHom_{\cor_{\V_\gamma}}(\cor_U,\roim{\phi_\gamma} \oim{\tau_w} F) \ar[d]^-{\wr}\\
\RHom_{\cor_{\V}}(\opb{\tau_v}\cor_U, F) \ar[d]^-{\wr} & \RHom_{\cor_{\V}}(\opb{\tau_w}\cor_U,F) \ar[d]^-{\wr}\\
\RHom_{\cor_{\V_\gamma}}(\cor_{U+v}, \roim{\phi_\gamma} F) \ar[r]^-{\Psi_F} & \RHom_{\cor_{\V_\gamma}}(\cor_{U+w}, \roim{\phi_\gamma}F).\\
}
\end{align*}
As $\roim{\phi_\gamma} \colon \Derb_{\gamma^{\circ\,a}}(\cor_\V) \to  \Derb(\cor_{\V_\gamma})$ is an equivalence of categories, the maps $(\Psi_F)_{F \in \Derb_{\gamma^{\circ\,a}}(\cor_\V)}$ provide  a natural transformation between the functors $\RHom_{\cor_{\V_\gamma}}(\cor_{U+v}, \cdot )$ and $\RHom_{\cor_{\V_\gamma}}(\cor_{U+w}, \cdot )$ by setting  for every $G \in \Derb(\cor_{\V_\gamma})$, $\xi_G:=\Psi_{\opb{\phi_\gamma}G}$. It follows from the enriched Yoneda lemma that the natural transformation $\xi$ is induced by the canonical map $\cor_{U+w} \to \cor_{U+v}$. Hence $\Phi_F$ is isomorphic to $\RHom_{\cor_{\V_\gamma}}(\cor_{U+v}, \chi^\gamma_{v,w} \roim{\phi_\gamma})$, which proves formula \eqref{iso:gammamicro}. 

Applying Formula \eqref{iso:gammamicro} to $\opb{\phi_\gamma}G$ and applying $\opb{\phi_\gamma}$ to both sides of the isomorphism, we obtain
\begin{equation*}
\opb{\phi_\gamma}\roim{\phi_\gamma} \chi^\mu_{v,w}(\opb{\phi_\gamma}G) \simeq \opb{\phi_\gamma} \chi^\gamma_{v,w} (\roim{\phi_\gamma}\opb{\phi_\gamma}G).
\end{equation*}
Finally, using that $\roim{\phi_\gamma} \opb{\phi_\gamma} \simeq \id$ and $\opb{\phi_\gamma} \roim{\phi_\gamma} \simeq \id$, we get the result.
\end{proof}

Let $F \in \Derb_{\gamma^{\circ\,a}}(\cor_{\V})$. Again, if $w=0$ and $v \in \gamma^a$, the morphism \eqref{mor:bismoothmu} provides the canonical morphism

\begin{equation*}\label{mor:smoothmu}
\chi^\mu_{v,0} \colon \oim{\tau_v}F \to  F.
\end{equation*}

\begin{remark}
Here, again, using Remark \ref{Rem:ab}, we obtain, for every $F \in \Mod_{\gamma^{\circ\,a}}(\cor_{\V_\gamma})$ and every $w\leq_\gamma v$, a canonical morphism $\oim{\tau_v}F \to \oim{\tau_w}F$ by setting $\chi^\mu_{v,w}(F):= \Hn^0(\chi^\mu_{v,w}(\iota(F))$.
\end{remark}

\subsection{Interleavings and distances}

Let $\cC$ be any of the following categories $\Der(\cor_{\V_\al})$, $\Der(\cor_{\V_\gamma})$, $\Derb_{\gamma^{\circ\,a}}(\cor_{\V_\gamma})$, $\Ch(\cor_{\V_\al})$, $\Ch(\cor_{\V_\gamma})$,  $\Mod(\cor_{\V_\al})$, $\Mod(\cor_{\V_\gamma})$, $\Mod_{\gamma^{\circ\,a}}(\cor_{\V_\gamma})$.

\begin{definition}
Let $F$, $G \in \cC$, and $v \in \gamma^a$. We say that $F$ and $G$ are $v$-interleaved if there exists $f \in \text{Hom}_\cC( \oim{\tau_v}F, G)$ and $g \in  \text{Hom}_\cC( \oim{\tau_v}G, F)$ such that the following diagram commutes.

\begin{equation*}\xymatrix{
\oim{\tau_{2v}}F  \ar[rd]\ar@/^0.7cm/[rr]^{\chi_{2v,0}(F)} \ar[r]^{\oim{\tau_v} f} & \oim{\tau_v} G  \ar[rd] \ar[r]^{g} & F \\
\oim{\tau_{2v}}G  \ar[ur]\ar@/_0.7cm/[rr]_{\chi_{2v,0}(G)} \ar[r]^{\oim{\tau_v} g} & \oim{\tau_v} F  \ar[ur] \ar[r]^{f} & G
   }
\end{equation*}
\end{definition}

\begin{definition}
With the same notations, define the interleaving distance between $F$ and $G$ with respect to $v\in \gamma^a$ to be : 
\begin{equation*}
d_I^v(F,G) := \inf(\{c \geq 0 \mid F ~\text{and}~G~\text{are}~c \cdot v-\text{interleaved}\} \cup \{ \infty \}).
\end{equation*}
\end{definition}

\begin{proposition}
The interleaving distance $d_I^v$ is a pseudo-extended metric on the objects of $\cC$, that is it satisfies for $F,G,H$ objects of $\cC$ :
\begin{enumerate}
    \item $d_I^v(F,G)\in \R_{\geq 0}\cup \{\infty\}$
    \item $d_I^v(F,G)=d_I^v(G,F)$
    \item $d_I^v(F,H)\leq d_I^v(F,G) + d_I^v(G,H)$
\end{enumerate}
\end{proposition}

\begin{notation}
We write $d_{I_\al}^v$ for the interleaving distance on $\Der(\cor_{\V_\al})$, $d_{I_\gamma}^v$ for the interleaving distance on $\Der(\cor_{\V_\gamma})$,  $d_{I_\mu}^v$ for the interleaving distance on $\Derb_{\gamma^{\circ\,a}}(\cor_{\V_\gamma})$. We write $d_{I^{\mathrm{ab}}_\al}^v$ for the interleaving distance on $\Mod(\cor_{\V_\al})$ and use similar notation in the cases of $\Mod(\cor_{\V_\gamma})$ and $\Mod_{\gamma^{\circ\,a}}(\cor_{\V_\gamma})$.
\end{notation}

\begin{remark}
Again, here we focus on the derived case as the abelian one can be deduced from the derived one by using Remark \ref{Rem:ab}. 

Note that, again, the results stated, in this paper, for the unbounded derived categories $\Der(\cor_{\V_\gamma})$ and $\Der(\cor_{\V_\al})$ also hold for their bounded counterparts $\Derb(\cor_{\V_\gamma})$ and $\Derb(\cor_{\V_\al})$ as they are full subcategories of the formers and all the functors considered have finite cohomological dimension (see Proposition \ref{prop:cohbound}) and the interleaving distances on $\Derb(\cor_{\V_\gamma})$ and $\Derb(\cor_{\V_\al})$ are equal to the restrictions of the interleaving distances on $\Der(\cor_{\V_\gamma})$ and $\Der(\cor_{\V_\al})$.
\end{remark}

\subsubsection{Interleavings and ephemeral modules}
This subsection is dedicated to the study of the relations between the notions of interleavings and ephemeral modules. We characterize ephemeral modules in terms of interleavings. Once again, we concentrate our attention on the derived setting as the abelian case can be deduced from the derived one by using Remark \ref{Rem:ab}.

\begin{proposition}\label{prop:inter}
Let $F$ and $G$ in $\Der(\cor_{\V_\al})$. The set
\begin{equation*}
\Inter(F,G) = \{v\in \Int(\gamma^a) \mid F~\text{and}~G~\text{are}~v-\text{interleaved}\}
\end{equation*}
is Alexandrov-closed.
\end{proposition}

\begin{proof}
It is sufficient to show that $\Inter(F,G)+\gamma^a=\Inter(F,G)$. The inclusion $\Inter(F,G) \subset \Inter(F,G)+\gamma^a$ is clear. We prove the reverse inclusion. Let $w \in \gamma^a$ and $v \in \Inter(F,G)$. Let
\begin{align*}
f \colon \oim{\tau_v} F \to G && g \colon \oim{\tau_v} G \to F
\end{align*}
be a $v$-interleaving between $F$ and $G$. The maps
\begin{align*}
\oim{\tau_{v+w}} F \stackrel{\oim{\tau_w}f}{\longrightarrow} \oim{\tau_w} G \stackrel{\chi^\al_{w,0}}{\longrightarrow} G && \oim{\tau_{v+w}} G \stackrel{\oim{\tau_w}g}{\longrightarrow} \oim{\tau_w} F \stackrel{\chi^\al_{w,0}}{\longrightarrow} F
\end{align*}
provides a $v+w$ interleaving between $F$ and $G$ since the following diagram
\begin{equation*}
\xymatrix@C=1.8cm@R=1.8cm{
\oim{\tau_{2(v+w)}} F \ar[r]^{\oim{\tau_{v+2w}}f} \ar@{=}[d]
&\oim{\tau_{v+2w}}G \ar[r]^{\chi^\al_{v+2w,v+w}} \ar@{=}[d]
&\oim{\tau_{v+w}}G \ar[r]^{\oim{\tau_\w}g}
& \oim{\tau_{w}}F  \ar[r]^{\chi_{w,0}}
& F \ar@{=}[d]\\
\oim{\tau_{2(v+w)}} F \ar[r]^{\oim{\tau_{2w}}\oim{\tau_{v}}f} \ar@/_2pc/[rrrr]_-{\chi^\al_{2(v+w),0}}
&\oim{\tau_{v+2w}}G \ar[r]^{\oim{\tau_{2w}}g} \ar[ru]_-{\chi^\al_{v+2w,v+w}}
&\oim{\tau_{2w}}F \ar[rr]^{\chi_{2w,0}} \ar[ru]_-{\chi^\al_{2w,w}}
&
& F
}
\end{equation*}
and its analogue with $F$ and $G$ interchanged are commutative. 
\end{proof}

\begin{corollary}
Let $v, w \in \gamma^a$ and assume that $v \geq_{\gamma^a} w$. Then,
\begin{equation*}
d_{I_\al}^v \geq d_{I_\al}^w. 
\end{equation*}
\end{corollary}

\begin{remark}
The proof of Proposition \ref{prop:inter} proves also that for $F, \;G \in \Der(\cor_{\V_\gamma})$, 
\begin{equation*}
\Inter(F,G) + \gamma^a = \Inter(F,G). 
\end{equation*}
Hence, if $v \geq_{\gamma^a} w$. Then, $d_{I_\gamma}^v \geq d_{I_\gamma}^w$.
\end{remark}

The following lemma is immediate.

\begin{lemma}\label{lem:ephzerointer}
Let $F\in \Der(\cor_{\V_\al})$ and $v \in \Int(\gamma^a) $. Then $F$ is $v$-interleaved with $0$ if and only if the canonical morphism $\chi^\al_{2v,0}(F) \colon \oim{\tau_{2v}}F \to F$ is null. \end{lemma}

\begin{proof}
If $\chi^\al_{2v,0}(F)$ is zero, it factors through zero and $F$ is $v$-interleaved with zero. The converse follows directly from the definition of a $v$-interleaving. 
\end{proof}

\begin{proposition}\label{prop:ephintercone}
Let $F\in \Mod(\cor_{\V_\al})$, then $F$ is ephemeral if and only if  
\begin{equation*}
\Inter(F,0) = \Int(\gamma^a).
\end{equation*}
\end{proposition}

\begin{proof}
\noindent (i) Assume $F$ is ephemeral. Let $v \in \Int(\gamma^a)$ and $U$ be an object of $\Op(\V_\al)$. We have the following sequence of inclusion
\begin{equation*}
U \subset U + \Int(\gamma) + v \subset U + 2v
\end{equation*}
and $U + \Int(\gamma) + v \in \Op(\V_\gamma)$. Hence 
$\Gamma(U + \Int(\gamma) + v ;F)\simeq 0$. It follows that $\chi^\al_{2v,0}(F) \colon \oim{\tau_{2v}}F \to F$  factors through zero. This implies that $v \in \Inter(F,0)$.

\noindent (ii) Assume that $\Inter(F,0)=\Int(\gamma^a)$. Let us show that $\oim{\beta} F \simeq 0$. It is sufficient to show that for every $x \in \V$, $F(x+\Int(\gamma))\simeq 0$. Let $x \in \V$. Then,
\begin{equation}\label{mor:nulliso}
\lim_{u+\gamma \subset x+\Int(\gamma)} \rho_{x+ \Int(\gamma),u} \colon F(x+\Int(\gamma)) \stackrel{\sim}{\to} \lim_{u+\gamma \subset x+\Int(\gamma)}F(u+\gamma).
\end{equation}
Let $u \in x+ \Int(\gamma)$, there exists $v \in  \Int(\gamma^a)$ such that $x=u+3v$ and by assumption
\begin{equation*}
\chi^\al_{2v,0}(F) \colon \oim{\tau_{2v}}F \to F
\end{equation*}
factor through zero. Thus, we have the following commutative diagram
\begin{equation*}
\xymatrix{ F(u+3v+\Int(\gamma)) \ar[r]^-{\chi^\al_{2v,0}(F)} \ar[rd]_-{\rho_{x+\Int(\gamma),u}}
& F(u+v+\Int(\gamma)) \ar[d]^-{\rho_{u+v+\Int(\gamma),u}}\\
&F(u+\gamma).
}
\end{equation*}
Hence, the restriction map $\rho_{x+ \Int(\gamma),u}$ is zero. This implies that the isomorphism \eqref{mor:nulliso} is null. It follows that $F(x+\Int(\gamma))\simeq 0$ which proves the claim.
\end{proof}

\begin{corollary}\label{cor:ephinter}
Let $F\in \Der(\cor_{\V_\al})$, then $F$ is ephemeral if and only if  
\begin{equation*}
\Inter(F,0) = \Int(\gamma^a).
\end{equation*}
\end{corollary}

\begin{proof}
\noindent (i) Assume $F$ is ephemeral and consider an homotopically injective replacement $\cI$ of it. Then considering $\cI$ as an object of $\Ch(\cor_{\V_\al})$ and noticing that step (i) of the proof of Proposition \ref{prop:ephintercone} extends to $\Ch(\cor_{\V_\al})$ proves the claim.

\noindent (ii) Assume that $\Inter(F,0)=\Int(\gamma^a)$. Then for every $i \in \Z$, $\Inter(\Hn^i(F),0)=\Int(\gamma^a)$. Then the results follow from Proposition \ref{prop:ephintercone}.
\end{proof}

\begin{corollary}\label{cor:effdis}
	Let $v \in \Int(\gamma^a)$. Then $F \in \Der_{\Eph}(\cor_{\V_\al})$ if and only if $d^v_{I_\al}(F,0)=0$.
\end{corollary}

\begin{proof}
	Assume $F \in \Der_{\Eph}(\cor_{\V_\al})$, then  $d^v_{I_\al}(F,0)=0$ is a direct consequence of Corollary \ref{cor:ephinter}.
	Let us prove the converse. Since $d^v_{I_\al}(F,0)=0$, for all $\varepsilon>0$, $\varepsilon\cdot v \in \Inter(F,0)$. Let us prove that $\Inter(F,0) = \Int(\gamma^a)$. Let $u \in \Int(\gamma^a)$. Since $\Int(\gamma^a)$ is open for the euclidean topology, there exists $\eta>0$ such that $u-\eta\cdot v \in \Int(\gamma^a)$. Remark that $u = \eta\cdot v + (u-\eta\cdot v)$. The first element of the sum belongs to $\Inter(F,0)$, and the second to $\Int(\gamma^a)$. By Proposition \ref{prop:inter}, $\Inter(F,0)$ is Alexandrov-closed, hence stable under addition by elements of $\Int(\gamma^a)$. This ends the proof.
\end{proof}

\subsubsection{Isometry theorems}

We prove that there is an isometry between the category of Alexandrov sheaves and the category of $\gamma$-sheaves both of them endowed with their respective version of the interleaving distance.

\begin{proposition}\label{prop:factorprop}
Let $F\in \Der(\cor_{\V_\al})$, then 
\begin{enumerate}[(i)]
\item $\Inter(F,\beta^{-1} \alpha^{-1}F)=\Int(\gamma^a)$,
\item $\Inter(F,\roim{\alpha}\oim{\beta}F)=\Int(\gamma^a)$.
\end{enumerate}
\end{proposition}

\begin{proof}

\noindent(i) We first prove that $\Inter(F,\beta^{-1} \alpha^{-1}F)=\Int(\gamma^a)$. Let $v \in \Int(\gamma^a)$, we first assume that $F \in \Ch(\cor_{\V_\al})$ the category of chain complexes of $\cor_{\V_\al}$-modules and remark that
\begin{equation*}
\oim{\tau_v}\opb{\beta} \opb{\alpha} F \simeq \opb{(\alpha \circ \beta \circ \tau_{-v})}F.
\end{equation*}
Let $U$ and $V$ be open subsets of $\V_\al$. As $(\alpha \beta \tau_{-v})^t(V)=V + \Int(\gamma)-v$, if $U \subset (\alpha \beta \tau_{-v})^t(V)$ then, $U \subset  (\alpha \beta \tau_{-v})^t(V) \subset V$. Hence, the restriction morphisms $F(V) \to F(U)$ provide a map
\begin{align*}
(\alpha \circ \beta \circ \tau_{-v})^\dagger F(U) \simeq \underset{U \subset \alpha \beta \tau_{-v}(V)}{\colim} F(V) \to F(U).
\end{align*}
Sheafifying, we get a map
\begin{equation*}
f\colon \oim{\tau_v}\opb{\beta} \opb{\alpha} F \to F.
\end{equation*} 

Let $U$ be an open subset of $\V_\al$ and let $v \in \Int(\gamma^a)$. Then $U \subset U+\Int(\gamma)+v$. Thus, by definition of colimits, there is a morphism
\begin{equation*}
F(U+v) \to \underset{U \subset V+\Int(\gamma)}{\colim} F(V).
\end{equation*}
This induces a morphism of sheaves
\begin{equation*}
g \colon \oim{\tau_v}F \to \opb{\beta}\opb{\alpha}F.
\end{equation*}

A straightforward computation shows that
\begin{align*}
\oim{\tau_{2v}} \opb{\beta}\opb{\alpha} F \stackrel{\oim{\tau_v}f}{\longrightarrow} F \stackrel{g}{\longrightarrow}  \oim{\tau_v} \opb{\beta}\opb{\alpha} F & \textnormal{\qquad and} & 
\oim{\tau_{2v}} F \stackrel{\oim{\tau_v}f}{\longrightarrow} \opb{\beta}\opb{\alpha} F \stackrel{g}{\longrightarrow}  \oim{\tau_v} F
\end{align*}
are respectively equals to the morphisms $\chi^{\al}_{2v,0}(\opb{\beta}\opb{\alpha} F)$ and $\chi^{\al}_{2v,0}(F)$.

If $F \in \Der(\cor_{\V_\al})$, the preceding construction applied to an homotopically injective replacement of $F$ provides an interleaving between $F$ and $\opb{\beta}\opb{\alpha} F$, as the functors $\oim{\tau_v}$, $\oim{\tau_{2v}}$, $\opb{\beta}$, $\opb{\alpha}$ are exact.\\

\noindent (ii) Let $v \in \Int(\gamma^a)$ and $\cI$ be an homotopically injective resolution of $F$. 
For every $U \in \Op(\V_\al)$, 
\begin{equation*}
U \subset \alpha^t(U) + v\subset U + v.
\end{equation*}
Hence, we get the morphisms of sheaves
\begin{align*}
f\colon\oim{\tau_v} \oim{\alpha}\oim{\beta} \cI \to  \cI && g\colon \oim{\tau_v}\cI \to \oim{\alpha}\oim{\beta} \cI.
\end{align*}
The morphisms $f$ and $g$ defines a $v$-interleaving between $\cI$ and $\oim{\alpha}\oim{\beta} \cI$. Hence, between $F$ and $\roim{\alpha}\oim{\beta} F$.
\end{proof}

\begin{corollary}
Let $F\in \Mod(\cor_{\V_\al})$, then 
\begin{enumerate}[(i)]
\item $\Inter(F,\beta^{-1} \alpha^{-1}F)=\Int(\gamma^a)$,
\item $\Inter(F,\oim{\alpha}\oim{\beta}F)=\Int(\gamma^a)$.
\end{enumerate}
\end{corollary}

\begin{lemma}\label{lem:prepiso}
Let $v\in \Int(\gamma^a)$ and denote by  $d_{I_\al}^v$ (resp. $d_{I_\gamma}^v$) the interleaving distance on $\Der(\cor_{\V_\al})$ (resp. $\Der(\cor_{\V_\gamma})$). Then : 
\begin{enumerate}[(i)]

   \item The functors $\roim{\alpha}$, $\opb{\beta}$ and $\oim{\beta}$ preserve $v$-interleavings,
      
   \item Let $F, \; G$ in $\Der(\cor_{\V_\gamma})$ then, $d_{I_\gamma}^v(F,G)=d_{I_\al}^v(\opb{\beta} F,\opb{\beta} G)=d_{I_\al}^v(\roim{\alpha} F,\roim{\alpha} G)$,

    \item  Let $F, \; G$ in $\Der(\cor_{\V_\al})$ then, $d_{I_\al}^v(F,G)=d_{I_\al}^v(\opb{\beta} \opb{\alpha} F,\opb{\beta} \opb{\alpha} G)$,
    
    \item  Let $F, \; G$ in $\Der(\cor_{\V_\al})$ then, $d_{I_\al}^v(F,G)=d_{I_\al}^v( \roim{\alpha} \oim{\beta} F, \roim{\alpha} \oim{\beta} G)$.
    
    \end{enumerate}
\end{lemma}

\begin{proof}
\begin{enumerate}[(i)]

\item This is a consequence of the fact that both morphisms of sites $\alpha$ and $\beta$ commute with $\tau_v$, combined with the isomorphisms \eqref{mor:smoothing_op_betastar}, \eqref{mor:smoothing_op_alphastar} and Lemma \ref{lem:beta_smooth}. 

\item This follows from the fully faithfulness of $\roim{\alpha}$ and $\opb{\beta}$ and that $\alpha$ and $\beta$ commute with $\tau_v$.

\item Using the triangular inequalities, we obtain
\begin{align*}
d_{I_\al}^v(F,G) & \leq d_{I_\al}^v(F,\opb{\beta} \opb{\alpha} F) + d_{I_\al}^v(\opb{\beta} \opb{\alpha} F,\opb{\beta} \opb{\alpha} G) + d_{I_\al}^v(\opb{\beta} \opb{\alpha} G, G)\\
& \leq d_{I_\al}^v(\opb{\beta} \opb{\alpha} F,\opb{\beta} \opb{\alpha} G)
\end{align*}
as $d_{I_\al}^v(F,\opb{\beta} \opb{\alpha} F)= d_{I_\al}^v(\opb{\beta} \opb{\alpha} G, G)=0$ by Proposition \ref{prop:factorprop}. Moreover, $\opb{\beta} \opb{\alpha}$ preserves interleaving. Hence,
\begin{equation*}
d_{I_\al}^v(\opb{\beta} \opb{\alpha} F,\opb{\beta} \opb{\alpha} G) \leq d_{I_\al}^v(F,G)
\end{equation*}
It follows that $d_{I_\al}^v(\opb{\beta} \opb{\alpha} F,\opb{\beta} \opb{\alpha} G)= d_{I_\al}^v(F,G)$.
\end{enumerate}
\end{proof}

\begin{theorem}\label{thm:isom}
Let $v\in \Int(\gamma^a)$, $F$, $G \in \Der(\cor_{\V_\al})$ and denote by  $d_{I_\al}^v$ (resp. $d_{I_\gamma}^v$) the interleaving distance on $\Der(\cor_{\V_\al})$ (resp. $\Der(\cor_{\V_\gamma})$). Then :
\begin{equation*}
d_{I_\al}^v(F,G)=d_{I_\gamma}^v(\oim{\beta}F,\oim{\beta}G).
\end{equation*}
\end{theorem}

\begin{proof}
By Lemma \ref{lem:prepiso} (i), $\oim{\beta}$ preserves $v$-interleavings. Hence, we obtain the inequality
\begin{equation*}
d_{I_\gamma}^v(\beta_*F, \beta_*G) \leq d_{I_\al}^v(F,G).
\end{equation*}

By Lemma \ref{lem:prepiso} (iii), $d_{I_\al}^v(F,G)= d_{I_\al}^v(\opb{\beta}\opb{\alpha}F,\opb{\beta}\opb{\alpha}G)$ and $\opb{\beta}$ preserves interleavings. Then,
\begin{equation*}
d_{I_\al}^v(\opb{\beta}\opb{\alpha}F,\opb{\beta}\opb{\alpha}G) \leq d_{I_\gamma}^v(\opb{\alpha}F,\opb{\alpha}G)
\end{equation*}
Finally, as $\oim{\beta}=\opb{\alpha}$,
\begin{equation*}
d_{I_\al}^v(F,G) \leq d_{I_\gamma}^v(\oim{\beta}F,\oim{\beta}G).
\end{equation*}
Hence, $d_{I_\al}^v(F,G)= d_{I_\gamma}^v(\oim{\beta}F,\oim{\beta}G)$.
\end{proof}

\begin{corollary}
Let $v \in \Int(\gamma^a)$ and $F \in \Der(\cor_{\V_\gamma})$. Assume that $d_{I_\gamma}^v(F,0)=0$. Then $F \simeq 0$.
\end{corollary}

\begin{proof}
	This result can be proved directly ``by hand". Here, we give a proof using our results. The functor $\oim{\beta}$ is essentially surjective. Therefore, there exists $G \in \Der(\cor_{\V_\al})$ such that $\oim{\beta}G\simeq F$. Moreover, $\oim{\beta}$ is an isometry by Theorem \ref{thm:isom}. Hence,
	\begin{equation*}
	d_{I_\al}^v(G,0)=d_{I_\gamma}^v(F,0)=0.
	\end{equation*}
	It follows from Corollary \ref{cor:effdis}, that $G$ is ephemeral. Thus, $F \simeq \oim{\beta}G \simeq 0$.
\end{proof}

Let $v \in \Int(\gamma^a)$, We write $d^v_{I_\mu}$ for the interleaving distance associated with $v$ on $\Derb_{\gamma^{\circ\,a}}(\cor_{\V})$.

\begin{proposition}
The functor $\roim{\phi_\gamma} \colon \Derb_{\gamma^{\circ\,a}}(\cor_{\V}) \to \Derb(\cor_{\V_\gamma})$ and its quasi inverse $\opb{\phi_\gamma}$ are isometries i.e. 
\begin{enumerate}[(i)]
\item for every $F, \;G \in \Derb_{\gamma^{\circ\,a}}(\cor_{\V})$, $d^v_{I_\mu}(F,G)= d^v_{I^\gamma}(\roim{\phi_\gamma}F,\roim{\phi_\gamma}G)$,

\item for every $F, \;G \in \Derb(\cor_{\V_\gamma})$, $d^v_{I_\gamma}(F,G)= d^v_{I_\mu}(\opb{\phi_\gamma}F,\opb{\phi_\gamma}G)$.
\end{enumerate}
\end{proposition}

\begin{proof}
First remark that the application $\phi_\gamma$ commutes with $\tau_v$ and that $\oim{\tau_v} \simeq \opb{\tau_{-v}}$. Finally, the result follows from Lemma \ref{lem:microgam}.
\end{proof}

A similar result was already proved in \cite{Berstableres}

\begin{lemma} Let $v \in \Int(\gamma^a)$, $\iota_\al \colon \Mod(\cor_{\V_\al}) \to \Der(\cor_{\V_\al})$ the functor which sends an object of $\Mod(\cor_{\V_\al})$ to the corresponding complex in degree zero, let $d^v_{I_\al^{\mathrm{ab}}}$ be the interleaving distance on $\Mod(\cor_{\V_\al})$ and $d^v_{I_\al}$ the interleaving distance on $\Der(\cor_{\V_\al})$. Then, for every $F,\; G \in \Mod(\cor_{\V_\al})$,
\begin{equation*}
d^v_{I_\al}(\iota(F),\iota(G))=d^v_{I_\al^{\mathrm{ab}}}(F,G).
\end{equation*}
\end{lemma}

\begin{proof}
Clear in view of Remark \ref{Rem:ab}.
\end{proof}

\begin{remark}
Similar results hold when replacing
\begin{enumerate}
\item $\iota_\al \colon \Mod(\cor_{\V_\al}) \to \Der(\cor_{\V_\al})$ by $\iota_\gamma \colon \Mod(\cor_{\V_\gamma}) \to \Der(\cor_{\V_\gamma})$ (resp. $\iota_\mu \colon \Mod_{\gamma^{\circ\,a}}(\cor_{\V_\gamma}) \to \Derb_{\gamma^{\circ\,a}}(\cor_{\V_\gamma})$,

\item  $\Mod(\cor_{\V_\al})$ by $\Mod(\cor_{\V_\gamma})$ (resp. $\Mod_{\gamma^{\circ\,a}}(\cor_{\V_\gamma})$),

\item $\Der(\cor_{\V_\al})$ by $\Der(\cor_{\V_\gamma})$ (resp. $\Derb_{\gamma^{\circ\,a}}(\cor_{\V_\gamma}))$,

\item $d^v_{I_\al^{\mathrm{ab}}}$ by $d^v_{I_\gamma^{\mathrm{ab}}}$ (resp. $d^v_{I_\mu^{\mathrm{ab}}}$),

\item $d^v_{I_\al}$ by $d^v_{I_\gamma}$ (resp. $d^v_{I_\mu}$).
\end{enumerate}
\end{remark}

\begin{corollary}\label{cor:isomab}
Let $v\in \Int(\gamma^a)$, $F$, $G \in \Mod(\cor_{\V_\al})$ and denote by  $d_{I_\al}^v$ (resp. $d_{I_\gamma}^v$) the interleaving distance on $\Mod(\cor_{\V_\al})$ (resp. $Mod(\cor_{\V_\gamma})$). Then :
\begin{equation*}
d_{I_\al}^v(F,G)=d_{I_\gamma}^v(\oim{\beta}F,\oim{\beta}G).
\end{equation*}
\end{corollary}

\section{Convolution and interleaving distances}

\subsection{Convolution distance}

We consider a finite dimensional real vector space $\V$ equipped with a norm $\|\cdot\|$ and $\cor$ a field. We endow $\V$ with the topology associated with the norm $\|\cdot\|$. Following \cite{Kashiwara2018a}, we briefly present the convolution distance. We first recall the following notations:

\begin{equation*}
s : \V \times \V \to \V, ~~~s(x,y) = x + y
\end{equation*}
\begin{equation*}
q_i : \V \times \V \to \V ~~(i=1,2) ~~~q_1(x,y) = x,~q_2(x,y) = y.
\end{equation*}

\begin{definition}\label{D:Convolution} The convolution bifunctor $\star \colon \Derb(\cor_\V)\times \Derb(\cor_\V) \to \Derb(\cor_\V)$ is defined by the formula : 
\begin{equation*}
F\star G = \text{R}s_!(F\boxtimes G).
\end{equation*}
\end{definition}

For $r \geq 0$, let $K_r := \cor_{B_r}$ with $B_r = \{x \in \V \mid  \| x \| \leq r \}$, seen as a complex concentrated in degree 0 in $\Derb(\cor_\V)$. For $r < 0 $, we set $K_r := \cor_{\{x \in \V \mid \, \|x\| < -r\}}[n]$ (where $n$ is the dimension of $\V$). 

The following proposition is proved in \cite{Kashiwara2018a}.

\begin{proposition}

\label{P:propertiesofconvolution} Let $r, r'\in \R$ and $F \in \Derb(\cor_\V)$. There are functorial isomorphisms 

\begin{equation*}
( K_{r}  \star K_{r'}) \star F \simeq K_{r + r'} \star F ~~~ and ~~~ K_0 \star F\simeq F.
\end{equation*}
\end{proposition}

If $r \geq r' \geq 0 $, there is a canonical morphism 
$\chi_{r, r'} \colon K_{r}\to K_{r'}$ in $\Derb(\cor_\V)$. It induces a canonical morphism $\chi_{r, r'} \star F \colon K_{r} \star F \to  K_{r'} \star F $. 
In particular when $r' = 0$, we get
\begin{equation}
\chi_{r,0} \star F \colon K_{r} \star F \to  F.
\end{equation}

Following \cite{Kashiwara2018a}, we recall the notion of $c$-isomorphic sheaves. 

\begin{definition}
Let $F,G \in \Derb(\cor_\V)$ and let $c \geq  0$. The sheaves $F$ and $G$ are c-isomorphic if there are morphisms $f : K_r  \star F \to G$ and $g : K_r \star G \to F$ such that the  diagrams

\begin{align*}
\xymatrix{ K_{2c} \star F \ar[rr]^-{{K_{2c}} \star f} \ar@/_2pc/[rrrr]_{{\chi_{2c,0}} \star F} && K_{c} \star G \ar[rr]^-{g} &&   F
},\\
\xymatrix{ K_{2c} \star G \ar[rr]^-{{K_{2c}} \star g} \ar@/_2pc/[rrrr]_{{\chi_{2c,0}} \star G} && K_{c} \star F \ar[rr]^-{f} &&   G
}.
\end{align*}
are commutative.
\end{definition}

The convolution distance for sheaves was introduce in \cite{Kashiwara2018a}. We recall its definition and refer to ibid. for more details concerning this pseudo-distance. The convolution distance is 
\begin{equation*}
d_\mathfrak{c}(F,G) := \inf(\{c \geq 0 \mid F ~\text{and}~G~\text{are}~c -\text{isomorphic}\} \cup \{ \infty \}).
\end{equation*}

\begin{remark}
The treatment of the interleaving and convolution distances can be unified through the notion of flow on a category (see \cite{deSilva2018}).
\end{remark}

\subsection{Comparison of the convolution and the interleaving distance}

We first review the notion of gauge (also called Minkowski functional) associated to a convex. We refer the reader to \cite[Ch. 15]{Rockafellar} for more details. 

\textbf{In all this subsection $\V$ is a finite dimensional real vector space endowed with a norm $\| \cdot \|$.}

\begin{definition}
Let $K$ a non-empty convex of $\V$ such that $0 \in \Int K$. The gauge of $K$ is the function 
\begin{equation*}
g_K \colon \V \to \R ,\; x \mapsto \inf \lbrace \lambda > 0 \,\vert  x \in \lambda K \rbrace.
\end{equation*}
\end{definition}

The following proposition is classic. We refer the reader to \cite[Theorem 15.2]{Rockafellar} for a proof.

\begin{proposition}
Let $K$ be a symmetric closed bounded convex subset of $(\V, \| \cdot \|)$ such that $0 \in \Int K$. Then $g_K$ is a norm on $\V$ .
\end{proposition}

Assume now that $\V$ is endowed  with a closed proper convex cone $\gamma$ with non-empty interior. Let $v \in \Int(\gamma^a)$ and consider the set 
\begin{equation*}
B_v := (v+\gamma) \cap (-v+\gamma^a).
\end{equation*}

\begin{lemma}
The set $B_v$ is a symmetric closed bounded convex subset of $\V$ such that $0 \in \Int B_v$.
\end{lemma}

\begin{proof}
The set $B_v$ is symmetric by construction and is closed and convex as it is the intersection of two closed convex sets. Since $v \in \Int (\gamma^a)$, there exists $\varepsilon > 0$ such that $B(v,\varepsilon) \subset \gamma^a$. Hence $B(0,\varepsilon)$ is a subset of $(v+\gamma)$ and $(-v+\gamma^a)$. This implies that $0 \in \Int B_v$.

Assume that $B_v$ is unbounded. Hence, there exists a sequence $(x_n)_{n \in \N}$ of points of $B_v$ such that $\| x_n \| \underset{n \infty}{\longrightarrow} \infty$. The sequence $(x_n / \| x_n \|)_{n \in \N}$ is valued in the the compact $\partial B(0,1)$. Thus, there is a subsequence $(\nu_n y_n)_{n \in \N}$ of $(x_n / \| x_n \|)_{n \in \N}$ with $\vert \nu_n \vert \underset{n \infty}{\longrightarrow} 0$ and such that for every $n \in \N$, $y_n \in B_v$ and $y_n$ converges to a limit $y$. By \cite[Theorem 8.2]{Rockafellar},
\begin{equation*}
y \in \lbrace z \in \V \, |\; \forall x \in B_v, \forall \lambda \geq 0,\; x+\lambda z \in B_v \rbrace.
\end{equation*}
Hence the half lines $-v+\R_{\geq 0}\, y$ and $v+\R_{\geq 0}\, y$ are contained in $B_v\subset v+ \gamma$. As $B_v$ is symmetric it follows that $v+\R_{\leq 0}\, y \subset B_v$. This implies that $\R \,y \subset \gamma$. This is absurd as $\gamma$ is a proper cone. Hence $B_v$ is bounded.
\end{proof}

It follows from the previous lemma that the gauge 
\begin{equation}\label{eq:normcool}
g_{B_v}(x)=\inf \lbrace \lambda > 0 \,\vert  x \in \lambda B_v \rbrace
\end{equation}
is a norm, the unit ball of which is $B_v$. From now on, we consider $\V$ equipped with this norm. \textbf{In the rest of this section, the ball are taken with respect to this norm}.

\begin{proposition}
Let $v \in \Int(\gamma^a)$, $c \in \R_{\geq 0}$ and $F, G \in \Derb_{\gamma^{\circ\,a}}(\cor_{\V})$. Assume that $\Supp(F)$ and $\Supp(G)$ are $\gamma$-proper subsets of $\V$. Then $F$ and $G$ are $c \cdot v$-interleaved if and only if they are $c$-isomorphic.
\end{proposition}

\begin{proof}
Let $F, G \in \Derb_{\gamma^{\circ\,a}}(\cor_{\V})$. Assume that $\Supp(F)$ and $\Supp(G)$ are $\gamma$-proper subsets of $\V$ and that they are $c \cdot v$-interleaved. We set $w=c \cdot v$. Hence, we have the maps 
\begin{align*}
\alpha \colon \tau_{w \ast} F \to G && \beta \colon \tau_{w \ast} G \to  F
\end{align*}
such that the below diagrams commute
\begin{align*}
\xymatrix{\tau_{2w \ast} F \ar[r]^-{\oim{\tau_w} \alpha} \ar@/_2pc/[rr]_{{\chi^\mu_{0,2w}}(F)} & \oim{\tau_{w}} G \ar[r]^-{\oim{\tau_{w}} \beta} &  F 
&&
\tau_{2w \ast} G \ar[r]^-{\oim{\tau_w} \beta} \ar@/_2pc/[rr]_{{\chi^\mu_{0,2w}}(G)} & \oim{\tau_{w}} F \ar[r]^-{\oim{\tau_{w}} \alpha} &  G. 
\\
}
\end{align*}
Using Lemmas \ref{lem:transker} and \ref{lem:microgam}, we obtain 
\begin{align*}
\xymatrix{\cor_{2w+\gamma^a} \sconv[np] F \ar[rr]^-{\cor_{2w+\gamma^a} \sconv[np] f} \ar@/_2pc/[rrrr]_{{\chi_{2w,0}} \sconv[np] F} && \cor_{w+\gamma^a} \sconv[np] G \ar[rr]^-{\cor_{w+\gamma^a} \sconv[np] g} &&   F.
}
\end{align*}
Hence, using the $\gamma$-properness of the supports of $F$, $G$ and that for every $c \geq 0$, 
\begin{equation*}
 \cor_{ c \cdot v+\gamma^a}\simeq \cor_{ c \cdot B_v+\gamma^a}\simeq K_{c} \star \cor_{\gamma^a}, 
\end{equation*} 
as well as Proposition \ref{prop:gammaloc}, we get
\begin{align*}
\xymatrix{K_{2c} \star F \ar[rr]^-{K_{c} \star f} \ar@/_2pc/[rrrr]_{{\chi_{2c,0}} \star F} && K_c \star G \ar[rr]^-{g} &&   F
}.
\end{align*}
Similarly we obtain the following commutative diagram
\begin{align*}
\xymatrix{K_{2c} \star G \ar[rr]^-{K_{c} \star g} \ar@/_2pc/[rrrr]_{{\chi_{2c,0}} \star G} && K_c \star F \ar[rr]^-{f} &&   G
}.
\end{align*}Hence, $F$ and $G$ are $c$ $\mathfrak{c}$-isomorphic. 

A similar argument proves that if $F$ and $G$ are $c$-isomorphic then they are $c\cdot v$-interleaved.
\end{proof}

\begin{corollary}\label{cor:gamconv}
Let $v \in \Int{\gamma^a}$, $F, G \in \Derb_{\gamma^{\circ\,a}}(\cor_{\V})$. Assume that $\Supp(F)$ and $\Supp(G)$ are $\gamma$-proper subsets of $\V$. Then
\begin{equation*}
d_\mathfrak{c}(F,G)=d^v_{I^\mu}(F,G)
\end{equation*}
where $d_\mathfrak{c}$ is the convolution distance associated with the norm $g_{B_v}$.
\end{corollary}

\end{document}